\documentclass[12pt,letter]{article}
\usepackage{amsmath}
\usepackage{amssymb}
\usepackage{amsthm}
\usepackage[left=2.5cm,right=2.5cm,top=2.5cm,bottom=2.5cm]{geometry}
\usepackage{bbm}
\usepackage{amsthm}
\usepackage{mathtools}
\usepackage[mathscr]{euscript}
\usepackage[utf8]{inputenc}
\usepackage{color}
\usepackage{enumerate}  
\usepackage{bm}
\usepackage[font=footnotesize,labelfont=bf]{caption}
\usepackage{subcaption}
\usepackage[shortlabels]{enumitem}
\usepackage[hidelinks]{hyperref}

\usepackage{graphicx}
\usepackage{tikz,pgfplots,ifthen}
\usetikzlibrary{calc}
\usepgfplotslibrary{polar}
\usepgfplotslibrary{fillbetween}
\usepackage{tikz-3dplot}
\usepackage{xcolor}
\usepackage[normalem]{ulem}

\newcommand{\R}[0]{\mathbb{R}}

\newcommand{\Z}[0]{\mathbb{Z}}

\renewcommand{\P}[0]{\mathbb{P}}
\newcommand{\E}[0]{\mathbb{E}}

\newcommand{\sE}[0]{\mathcal{E}}

\newcommand{\sH}[0]{\mathcal{H}}

\newcommand{\sC}[0]{\mathcal{C}}

\newcommand{\sS}[0]{\mathcal{S}}

\newcommand{\sP}[0]{\mathcal{P}}
\newcommand{\sR}[0]{\mathcal{R}}
\newcommand{\sU}[0]{\mathcal{U}}
\newcommand{\sV}[0]{\mathcal{V}}
\newcommand{\sL}[0]{\mathcal{L}}

\newcommand{\Var}[0]{\text{Var\,}}
\newcommand{\Cov}[0]{\text{Cov\,}}

\newcommand{\m}[0]{\wedge}

\newtheorem{theorem}{Theorem}[section]
\newtheorem{corollary}[theorem]{Corollary}
\newtheorem{lemma}[theorem]{Lemma}
\newtheorem{prop}[theorem]{Proposition}
\newtheorem{definition}[theorem]{Definition}
\newtheorem{remark}[theorem]{Remark}

\newtheorem*{conditions}{Conditions}

\allowdisplaybreaks
\numberwithin{equation}{section}

\newcommand{\bl}[0]{{\bf l}}

\newcommand{\eps}{\varepsilon}

\makeatletter
\renewenvironment{proof}[1][\proofname]{\par
\pushQED{\qed}%
\normalfont \topsep6\p@\@plus6\p@\relax
\trivlist
\item\relax
{\bfseries  
#1\@addpunct{.}}\hspace\labelsep\ignorespaces 
}{%
\popQED\endtrivlist\@endpefalse
}
\makeatother

\title{Poisson statistics, vanishing correlations, and extremal particle limits for symmetric exclusion in $d > 1$} 

\author{Michael Conroy\thanks{Clemson University, meconro@clemson.edu} \ \ and \ \ Sunder Sethuraman\thanks{University of Arizona, sethuram@arizona.edu}}

\date{}

\begin{document}
\maketitle

\begin{abstract} We consider the symmetric simple exclusion system on $\Z^d$, $d \ge 2$, starting from a class of ``step'' initial
conditions in which particles are constrained within a half-space. One may count the number $N_t$ of particles that have moved beyond a distance $z = z(t)$ into the initially-empty half of $\Z^d$ at time $t$. We show in large generality that when $\lim_{t\to\infty} E[N_t]$ exists, correlations between particles beyond $z$ 
vanish as $t \to \infty$ so as to allow convergence of $N_t$ to the same Poisson distribution one would get were the particles allowed to move independently. 
When the initial condition constrains a region of polynomial growth, we identify $z(t)$ and the limit of $E[N_t]$ explicitly. As a consequence of the limit, we obtain a Gumbel limit distribution for the extremal particle position, as well as the limiting distributions of all order statistics.
\end{abstract}

\smallskip
{\small

}

\section{Introduction}

 Informally, the symmetric simple exclusion process (SSEP) on $\Z^d$ consists of a system of continuous time simple 
random walks, in which particles can jump to each neighboring lattice point with probability $(2d)^{-1}$ at rate $1$, but where jumps to already-occupied sites are suppressed. A more formal definition in terms of both a generator and the related ``stirring'' process is given in Section \ref{defn_section}.  Such an interacting particle system is well-known in the modeling of traffic, queues, fluid flow, and other applications.  See for instance books \cite{KL}, \cite{LigBook05}, \cite{LigBook}, and \cite{Spohn} for a more in depth discussion of its history and motivations.  Recent investigations in the symmetric exclusion process include \cite{EFX} for joint occupation time and current fluctuations in $d=1$, \cite{xue-zhao} for a moderate deviation principle from the hydrodynamic limit in $d=1$, \cite{redig-saada} for effects of a ``source'' on the stationary distribution, \cite{Gess} for quantitative central limit theorems, and \cite{Bodineau} for large deviations for two point correlations in $d=1$, among others.

While the symmetric exclusion process is a well-established model, the study of its extreme value behavior, although a natural concern, is not well understood.  Part of the difficulty is that the asymptotics of the extremes is beyond the diffusive scale, and not captured by the hydrodynamic bulk mass limit.  A general purpose of this work is to study the asymptotic behavior of the extreme values in dimensions $d>1$ including connections with certain geometries, following a previous investigation in $d=1$ where the geometrical influence is more limited.

Let $\{\eta_t(k): k\in \Z^d\}$ denote the occupation variables in the exclusion system, so that $\eta_t(k)=1$ when there is a particle at $k$ at time $t$, and $\eta_t(k) = 0$ otherwise.  Given a ``level'' $z\geq 0$, consider the hyperspace $B_z=\{k\in \Z^d: k_1 > z\}$ and the number of particles $N_t = \sum_{k\in B_z}\eta_t(k)$ in it.  Let $\mathcal{P}_t = \{k_1\in \Z: \eta_t(k) = 1\}$ be the projection of occupied sites on the first coordinate, and let $X_t = \max \mathcal{P}_t$ be the maximal position of a particle in the $k_1$ direction.  

Roughly, one can understand the $k_1$-locations of the extreme values in terms of $z=z(t)$ if $\{N_t: t\geq 0\}$ is tight.  Moreover, if $N_t$ converges weakly to some random variable $Z$, then we could conclude $P(X_t\leq z) = P(N_t = 0) \rightarrow P(Z=0)$.  A fuller distributional limit for $X_t$ may be obtained when the level $z$ is allowed to vary with respect to a parameter, say $x\in \R$.  In this case, the limit $Z=Z_x$ may give a finer description to the statistics of $X_t$.

For instance, in the one dimensional setting $d=1$, it has been shown under a step initial condition, that is when $\eta_0(k) = 1$ for $k\leq 0$ and $\eta_0(k)=0$ for $k>0$, that $z$ may be taken in the form $z = b^0_t(x+a^0_t)$, where
$a^0_t = \log(t/(\sqrt{2\pi}\log t))$ and $b^0_t = \sqrt{t/\log t}$, so that $N_t$ converges weakly to a Poisson random variable $Z_x$ with mean $e^{-x}$.  Hence, the extreme particle location satisfies a Gumbel limit: $P(X_t/b^0_t - a^0_t \leq x)\rightarrow P(Z_x=0) = 
e^{-e^{-x}}$; see \cite{ConSet23}.

Interestingly, the same limits hold in independent particle systems on $\Z$, where there is no correlation between particles.
For instance, we recall the standard result that, for a row-wise independent array $\{\chi_{i,t}\}$ of Bernoulli random variables, 
$\sum_{i=1}^\infty \chi_{i,t} \Rightarrow \mbox{Poisson}(\lambda)$
as $t \to \infty$
if and only if 
\begin{equation}\label{eq:meanandss}
	\sum_{i=1}^\infty E[\chi_{i,t}] \to \lambda \qquad\mbox{and}\qquad \sum_{i=1}^\infty (E[\chi_{i,t}])^2 \to 0. 
\end{equation}
Moreover, since $\sum_{i=1}^\infty (E[\chi_{i,t}])^2 = E[N_t] - \Var(N_t)$, the line 
\eqref{eq:meanandss} is equivalent to 
\begin{equation}\label{eq:meanvariancelimit}
	\lim_{t\to\infty} E[N_t] = \lim_{t\to\infty} \Var(N_t) = \lambda. 
\end{equation}
One may apply these criteria to $N_t$ in the independent particle system context \cite{Arr83}.   Further, the second limit in \eqref{eq:meanandss} is implied by the first, and so only convergence of $E[N_t]$ is required for the Poisson limit of $N_t$.  See also \cite{Corwin} and \cite{MikYsl20} for recent treatments of the extremes in systems of independent particles in different settings. 

However, the proof of the Poisson and Gumbel limits in $d=1$ with respect to the symmetric exclusion process involve making correlation estimates of the exclusion particles.  Indeed, by the Strong Rayleigh property of symmetric exclusion processes (see Subsection \ref{Rayleigh}), the Poisson($\lambda$) limit of $N_t$ holds exactly when \eqref{eq:meanvariancelimit} holds.  In particular,   
 \begin{align}\label{eq:meanminusvarexcl}
	E[N_t] - \Var(N_t) = \sum_{x} (E[\eta_t(x)])^2 - \sum_{x \ne y} \Cov(\eta_t(x), \eta_t(y)),
\end{align}
and the second term---not present with respect to the independent particle system---was shown to vanish in $d=1$ via duality properties of the symmetric exclusion process and precise estimates of random walk probabilities, thereby proving the Poisson limit.

At this point, one might ask in the symmetric exclusion setting if finding appropriate levels $z$ where $E[N_t]$ has a limit is in itself enough to deduce \eqref{eq:meanminusvarexcl} and therefore a distributional limit for $N_t$ and also for $X_t$, as it would be for the independent particle system.
A main aim of this article is to study this question for the $d>1$ SSEP. 

There are several questions of interest to address.  What are suitable initial ``step'' profiles where tightness levels $z$ can be determined?   How does the strength of correlations depend on dimension and the initial condition, and are there usable bounds?  Is convergence of $E[N_t]$ indeed sufficient for Poisson convergence of $N_t$?  When can one find levels $z$ depending on a parameter so that a Gumbel limit for $X_t$ holds? 

In this article we introduce a notion of a higher-dimensional step initial condition
(see Sections \ref{discussion} and \ref{defn_section}) that leads to a well-defined 
problem.   We observe there will be several types of initial conditions in the $d>1$ context, given by geometric profiles more complex than the half-filled line condition in $d=1$. 
With respect to such initial conditions, when $d \ge 4$, we show 
that convergence of $E[N_t]$ is sufficient to deduce that particle correlations vanish and $N_t$ has a Poisson limit.  In $d = 2,3$, we give an explicit additional geometric condition on the step profile so that this takes place
(Theorem \ref{meta}, Corollary \ref{poissonlimittheorem}).
We conclude therefore, with respect to this class of initial conditions in $d\geq 2$, that the Poisson limit for $N_t$ holds for the symmetric exclusion system when it holds for independent particles

One may attribute these dimension-dependent results, in a certain sense, to less rigidity and more room for exclusion particles to spread apart in higher dimensions, translating to less dependence among particles.
On this point, in our analysis of particle covariances (Proposition \ref{mainbound}), we find quantitative bounds, depending on the initial profile and the dimension $d\geq 2$, that are easier to handle and reveal themselves better in terms of $E[N_t]$, especially when $d\geq 4$. 

A more concrete aim of this work is to identify levels of the form $z = b_t(x+a_t)$ 
via which we may show Gumbel limits for $X_t$. 
Indeed, with respect to a large class of polynomially-shaped initial conditions in $d\geq 2$, we 
verify in Section \ref{gumbelsec} the geometric condition of Corollary \ref{poissonlimittheorem} in $d=2,3$, and thereby show the desired limit in $d\geq 2$; see Theorem \ref{pyramidmain}.

\medskip

\noindent {\bf Organization of the paper.}  In Section \ref{in depth section}, we give a more in-depth account of our results and proof methods, and we list some open problems for the interested reader.  
Section \ref{defn_section} contains formal definitions and collects relevant properties of the symmetric exclusion process.  Precise statements of thee main results are presented in Section \ref{mainresults}, with an application to limit distributions of extreme values given in Section \ref{gumbelsec}. The main proofs are organized into Sections \ref{helping-section}, \ref{mainthmproof}, and \ref{mainboundpf}. Section \ref{rwlemmas} is an Appendix containing staight-forward proofs of various random walk and Gaussian asymptotic results we use in previous sections.

\section{Discussion of results}
\label{in depth section}

In the following subsections, we describe informally what we mean by a ``step profile'' in $d\geq 2$, followed by a more detailed discussion of our results and proof methods.  In the last subsection, we give a few open problems.

\subsection{A step profile in higher dimensions}
\label{discussion}
What is an appropriate analogue of a step initial profile in higher dimensions? While placing a particle initially at every point in a half space is natural in $d=1$, this leads to triviality in $d \ge 2$. Indeed, consider SSEP in $d = 2$, and suppose initially the occupation variables $\{\eta_t(k)\}$ satisfy
$\eta_0(k) = 1(k_1 \le 0)$, for $k = (k_1, k_2) \in \Z^2$. 
That is, we have placed a particle at every lattice point on and to the left of the $k_2$-axis. 

We may ask how many particles $N_t$, when projected onto the $k_1$-direction, have moved beyond a value $z > 0$ at a finite time $t$. Of course this is infinite, as the first of infinitely many independent Exponential$(1)$ clocks will ring instantaneously. To be more precise, consider the ``stirring'' representation of $N_t$, namely  
$N_t = \sum_{k\in \Z^2, k_1\leq 0}1(\xi_k(t) \cdot (1,0) > z)$.
	
Here, $\{\xi_k(t) : k \in \Z^2\}$ is a collection of dependent random walks such that marginally
each $\xi_k(\cdot)$ is a simple random walk on $\Z^2$ with $\xi_k(0)=k$, and its projection on the $k_1$-direction does not depend on its $k_2$ value (see Section \ref{stir} for this construction). 
 Letting $e_1 = (1,0)$, we have 
\[
	E[N_t] = \sum_{k \in \Z^2, k_1\le0} P(\xi_k(t) \cdot e_1 > z) = \sum_{k_1 \le 0} \sum_{k_2 \in \Z} P(\xi_{(k_1,0)}(t) \cdot e_1  > z) = \infty. 
\]
. 
As discussed in more detail in Section \ref{Rayleigh}, the indicator variables
$\{1(\xi_k(t) \cdot e_1 \le z) : k \in \Z^2\}$ are negatively associated for each $t$. Then for the position $X_t$ of the right-most (in the $k_1$-direction) particle, 
\begin{align*}
	P(X_t \le z) = P(N_t = 0) &= P(\xi_k(t) \cdot e_1 \le z \;\,\mbox{for all}\;\, k \;\,\mbox{with}\; \, k_1 \le 0) \\
	&\le \prod_{k \in \Z^2, k_1 \le 0} P(\xi_k(t) \cdot e_1 \le z) \\
	&= \exp \Big( \sum_{k \in \Z^2, k_1 \le 0} \log ( 1 - P(\xi_k(t) \cdot e_1  > z) ) \Big) \le \exp\left( - E[N_t] \right) = 0. 
\end{align*}
It follows that $P(X_t = \infty) = 1$ for any positive $t$. 

We construct initial conditions for which $E[N_t] < \infty$ by placing particles in subsets of the half space $\sH_d = \{k \in \Z^d : k_1 \le 0\}$ determined by ``shape'' functions $g_2, \ldots, g_d : [0,\infty) \to [0, \infty)$, namely 
\[
	\eta_0(k) = 1(k \in \sH_d : |k_i| \le g_i(-k_1)\;\,\mbox{for}\;\, i = 2, \ldots, d). 
\]
An example of such an initial profile in $d = 2$ is depicted in Figure \ref{fig:generalprofile} (a). 
A more degenerate case allowed is when $g_i \equiv 0$ for all $i$, corresponding to a single line of particles at points $(k_1, 0, \ldots, 0)$, $k_1 \le 0$.

Under conditions on $\{g_i\}$ (see Section \ref{defn_section}) and the appropriate scaling $z$, 
\begin{equation}\label{eq:ENasympF}
	E[N_t] \asymp \sum_{j \ge 0} \prod_i g_i(j) P(\zeta_{t/d} > z + j) \asymp E\Big[ \int_0^{(\zeta_{t/d} - z)_+} \prod_i g_i(u)\,du\Big] , 
\end{equation}
where $\zeta_t$ is a simple symmetric random walk on $\Z$ starting from $0$ (Lemma \ref{meanasymp}). Thus $P(X_t \in \Z) = 1$ for finite $t$ whenever $E[\int_0^{(\zeta_t)_+}\prod_ig_i(u)\,du] < \infty$. 
We briefly illustrate the case of linear shape functions $\{g_i\}$ the next section.

Under these initial conditions, there is always an infinite number of particles in the system, so that the appropriate scaling $z$ for $X_t$ will be superdiffusive.  As in $d=1$, the behavior of $X_t$ is beyond the diffusive scale of ``bulk'' particle mass hydrodynamics.

\begin{figure}[t]
\captionsetup{width=.95\linewidth}
\centering
\hspace*{\fill}%
\begin{subfigure}[b]{0.3\textwidth}
\centering
\begin{tikzpicture}[scale=0.7]

\draw[->] (-6.5,0) -- (1,0) node[right] {\footnotesize $k_1$};
\draw[->] (0,-3.5) -- (0,3.5) node[above] {\footnotesize $k_2$};

\draw[scale=0.5, domain=-11.75:0, smooth, variable=\x, thick] plot ({\x},{0.04*\x*\x + 0.25*sin(deg(\x-1.5)) +  1.8});
\draw[scale=0.5, domain=-11.75:0, smooth, variable=\x, thick] plot ({\x},{-0.04*\x*\x - 0.25*sin(deg(\x-1.5)) -  1.8});

\foreach \i in {-6.5,-6,...,0}{
	\node[draw,circle,inner sep=1.5pt] at (\i,0) {};
	\node[draw,circle,inner sep=1.5pt] at (\i,-0.5) {};
	\node[draw,circle,inner sep=1.5pt] at (\i,0.5) {};
	}

\foreach \i in {-6.5,-6,...,-1}{
	\node[draw,circle,inner sep=1.5pt] at (\i,1) {};
	\node[draw,circle,inner sep=1.5pt] at (\i,-1) {};
	}

\foreach \i in {-6.5,-6,...,-3}{
	\node[draw,circle,inner sep=1.5pt] at (\i,1.5) {};
	\node[draw,circle,inner sep=1.5pt] at (\i,-1.5) {};
	}

\foreach \i in {-6.5,-6,...,-4}{
	\node[draw,circle,inner sep=1.5pt] at (\i,2) {};
	\node[draw,circle,inner sep=1.5pt] at (\i,-2) {};
	}

\foreach \i in {-6.5,-6,...,-4.5}{
	\node[draw,circle,inner sep=1.5pt] at (\i,2.5) {};
	\node[draw,circle,inner sep=1.5pt] at (\i,-2.5) {};
	}

\foreach \i in {-6.5,-6,...,-5}{
	\node[draw,circle,inner sep=1.5pt] at (\i,3) {};
	\node[draw,circle,inner sep=1.5pt] at (\i,-3) {};
	}
	
\foreach \i in {-6.5,-6}{
	\node[draw,circle,inner sep=1.5pt] at (\i,3.5) {};
	\node[draw,circle,inner sep=1.5pt] at (\i,-3.5) {};
	}

\end{tikzpicture}
\caption{}
\end{subfigure}\hfill%
\begin{subfigure}[b]{0.3\textwidth}
\centering
 \tdplotsetmaincoords{68.444444}{-344.3333}
\begin{tikzpicture}[x=1cm,y=1cm,z=1cm,scale=1, tdplot_main_coords,scale=0.75]
	\pgfmathsetmacro{\a}{-3}
	\pgfmathsetmacro{\bs}{0.75}
	\pgfmathsetmacro{\bl}{2.5}
	\coordinate (A1) at (0,\bs,\bs) {};
	\coordinate (A2) at (0,\bs,-\bs) {};	
	\coordinate (A3) at (0,-\bs,-\bs) {};	
	\coordinate (A4) at (0,-\bs,\bs) {};	
	\coordinate (B1) at (\a,\bl,\bl) {};
	\coordinate (B2) at (\a,\bl,-\bl) {};	
	\coordinate (B3) at (\a,-\bl,-\bl) {};	
	\coordinate (B4) at (xyz cs:x=\a,y=-\bl,z=\bl) {};

	\draw[-,thick] (A1) -- (A2) -- (A3) -- (A4) -- cycle;
	\draw[-, thick] (B1) -- (B2) -- (B3) -- (B4) -- cycle;
	\draw[-, thick] (A1) -- (B1);
	\draw[-,thick] (A2) -- (B2);
	\draw[-,thick] (A3) -- (B3);	
	\draw[-,thick] (A4) -- (B4);

	\draw[->] (xyz cs:x=-4.5) -- (xyz cs:x=1) node[right] {\footnotesize $k_1$};
	\draw[->] (xyz cs:z=-3.5) -- (xyz cs:z=3.5) node[above] {\footnotesize $k_2$};
	\draw[->] (xyz cs:y=-3.5) -- (xyz cs:y=3.5) node[above] {\footnotesize $k_3$};

	\foreach \i in {-2.5,-2,...,2.5}{
		\foreach \j in {-2.5,-2,...,2.5}{
			\node[draw,circle,inner sep=1.5pt] at (\a,\i,\j) {};
		}
	}
	
	\foreach \i in {-2,-1.5,...,2}{
		\foreach \j in {-2,-1.5,...,2}{
			\node[draw,circle,inner sep=1.5pt] at (\a+0.5,\i,\j) {};
		}
	}

	\foreach \i in {-1.5,-1,...,1.5}{
		\foreach \j in {-1.5,-1,...,1.5}{
			\node[draw,circle,inner sep=1.5pt] at (\a+1,\i,\j) {};
		}
	}

	\foreach \i in {-1.5,-1,...,1.5}{
		\foreach \j in {-1.5,-1,...,1.5}{
			\node[draw,circle,inner sep=1.5pt] at (\a+1.5,\i,\j) {};
		}
	}

	\foreach \i in {-1,-0.5,...,1}{
		\foreach \j in {-1,-0.5,...,1}{
			\node[draw,circle,inner sep=1.5pt] at (\a+2,\i,\j) {};
		}
	}

	\foreach \i in {-1,-0.5,...,1}{
		\foreach \j in {-1,-0.5,...,1}{
			\node[draw,circle,inner sep=1.5pt] at (\a+2.5,\i,\j) {};
		}
	}

	\foreach \i in {-0.5,0,0.5}{
		\foreach \j in {-0.5,0,0.5}{
			\node[draw,circle,inner sep=1.5pt] at (\a+3,\i,\j) {};
		}
	}

	\foreach \i in {-2.5,-2,...,2.5}{
		\foreach \j in {-2.5,-2,...,2.5}{
			\node[draw,circle,inner sep=1.5pt] at (\a-0.5,\i,\j) {};
		}
	}

	\foreach \i in {-2.5,-2,...,2.5}{
		\foreach \j in {-2.5,-2,...,2.5}{
			\node[draw,circle,inner sep=1.5pt] at (\a-1,\i,\j) {};
		}
	}

\end{tikzpicture}
\caption{}
\end{subfigure}%
\hspace*{\fill}%
\caption{\small (a) An arbitrary initial profile in $\Z^2$ determined by a nonnegative ``shape'' function. (b) An initial profile in $\Z^3$ determined by two linear functions.}\label{fig:generalprofile}
\end{figure}
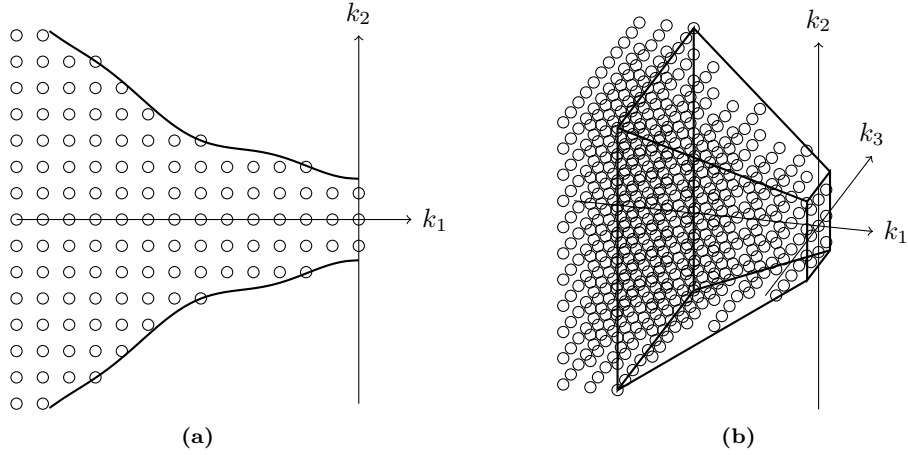

\subsection{Overview of results}\label{resultsoverview}
As alluded to in the Introduction, to investigate the extent to which convergence of $E[N_t]$ is enough for Poisson convergence of $N_t$, given initial shape functions $\{g_i\}$, we seek upper bounds on the quantities 
\[
	\sS_t(\{g_i\},z) = \underset{k_1 > z}{\sum_{k \in \Z^d} }(E[\eta_t(k)])^2 \qquad \mbox{and} \qquad \sC_t(\{g_i\},z) = - 2\underset{j_1,k_1 > z}{\sum_{\{j,k\} \subset \Z^d}} \Cov(\eta_t(j), \eta_t(k))
\]
that vanish in the $t\to\infty$ limit whenever $\sup_t E[N_t] < \infty$. (The sum of these objects equals $E[N_t]-\Var(N_t)$ as in the one dimensional case \eqref{eq:meanminusvarexcl}; see the later discussion in Section \ref{Rayleigh}.)

To this end, we obtain bounds on $\sS_t(\{g_i\},z)$ and $\sC_t(\{g_i\},z)$ (Propositions \ref{ssbound}, \ref{mainbound}) for nondecreasing, continuously differentiable functions $\{g_i\}$ whose derivatives satisfy a regularity condition that uniformly limits the variation of the functions on unit intervals
(Conditions \ref{nondecrcond} and \ref{covcond} in Section \ref{defn_section}). 
These assumptions are not overly restrictive on allowable initial profiles due to the monotonicity 
\[
	\sS_t(\{g_i\},z) + \sC_t(\{g_i\},z) \le \sS_t(\{f_i\},z) + \sC_t(\{f_i\},z), 
\]
when $g_i \le f_i$ for each $i$ (Lemma \ref{monotoneerror}). So, less well-behaved initial profiles can also be analyzed by considering suitable nondecreasing, differentiable dominating functions  (see Remark \ref{mainremark}). 

While these bounds hold for a large class of profiles, 
they become useful when the growth of the shape functions 
is further limited to being sub-exponential (Condition \ref{smallderiv}). With this added condition, we show that $\sS_t(\{g_i\},z), \sC_t(\{g_i\},z) \to 0$ in large generality when $d \ge 4$. In dimensions $2$ and $3$, additional consideration of the scaling $z$ and the shape $\{g_i\}$ is required; see Theorem \ref{meta}.

More precisely, when in addition the logarithmic derivative $(\prod_i g_i)'/\prod_i g_i$ vanishes at infinity, then via \eqref{eq:ENasympF}, Proposition \ref{ssbound}, and Corollary \ref{Ccor}, the condition $\sup_t E[N_t] < \infty$ implies that
\[
	\sS_t(\{g_i\},z) = O\Big(E\Big[ \prod_i g_i(\zeta_{t/d} - z)1(\zeta_{t/d} > z) \Big] \Big) \to 0 \qquad\text{as}\qquad t\to\infty. 
\]

The analysis of $\sC_t(\{g_i\},z)$ is more involved than that of $\sS_t(\{g_i\},z)$. 
In Corollary \ref{Ccor}, when Conditions (A)--(C) hold and $\sup_tE[N_t] < \infty$, we find
\begin{equation}\label{eq:introC}
\begin{aligned}
	\sC_t(\{g_i\},z) &= O\Big( \gamma_d(t) \Big( E\Big[ \prod_i g_i(\zeta_{t/d} - z)1(\zeta_{t/d} > z) \Big]  \Big)^2 \\
	&\qquad\quad + \gamma_{d+1}(t) E\Big[ \prod_i g_i(\zeta_{t/d} - z)1(\zeta_{t/d} > z) \Big]  \Big), 
\end{aligned}
\end{equation}
as $t \to \infty$, 
where 
\begin{align}
\label{gamma-d}
	\gamma_d(t) = \left.\begin{cases} \sqrt{t} &\;\text{if}\;\, d=2 \\ \log t &\;\text{if}\;\,  d=3 \\ 1 &\;\text{if}\;\,  d \ge 4 \end{cases} \right\}
	\asymp \int_0^t P(\zeta_s^{(d-1)} = 0 | \zeta_0^{(d-1)} = 0 )\,ds.
\end{align}
Above, $\zeta^{(d-1)}$ is a simple random walk on $\Z^{d-1}$, and so $\gamma_d(t)$ is the order of expected number of returns to $0$ for the random walk in dimension $d-1\geq1$. The $d-1$ dimensions correspond to those in which movement of particles does not affect their positions projected onto the first coordinate. Thus higher dimensions yield sharper bounds on $\sC_t(\{g_i\},z)$: In some sense, recurrence in $d-1 \in \{1,2\}$ means a lack of space for particles to spread apart. Hence when $\sC_t(\{g_i\},z) \to 0$, it does so at a slower rate of convergence in these dimensions. 

Although these estimates on $\sS_t(\{g_i\}, z)$ and $\sC_t(\{g_i\}, z)$ are general, their application with respect to a shape $\{g_i\}$ depends on finding a level $z$ so that the limit of $E[N_t]$ exists.  In $d=2,3$, vanishing of the second term in the bound \eqref{eq:introC} will follow from finiteness of $E[N_t]$ and Conditions (A)--(C), but the vanishing of the first term, stated in \eqref{eq:d=2,3cond} of Theorem \ref{meta}, may be thought of as an additional condition on $\{g_i\}$ for the Poisson limit of $N_t$ to hold. In $d\geq 4$ it holds automatically as \eqref{eq:introC} is on the same order as $\sS_t(\{g_i\},z)$, which vanishes regardless. 

Nevertheless, for a class of polynomial shapes $\{g_i\}$, we are able to identify $z$ (depending on a parameter $x$---see Section \ref{gumbelsec}) and the associated convergences (Theorem \ref{pyramidmain}) in all $d\geq 2$.  To illustrate these results,
we discuss here the example of linear shape functions $g_i(u) = c_iu + r_i$ with $c_i > 0$, $r_i \ge 0$ (this is depicted for $d = 3$ in Figure \ref{fig:generalprofile} (b)). 

In this case, \eqref{eq:introC} reduces to 
\begin{align*}
	\sC_t(\{g_i\},z) &= 
	O\big( \gamma_d(t) ( E[ (\zeta_{t/d} - z)_+^{d-1} ] )^2 \big), \qquad t\to\infty. 
\end{align*}
Indeed, in this example, \eqref{eq:ENasympF} becomes $E[N_t] \asymp E[(\zeta_{t/d} - z)_+^d]$. Thus from an application of H\"older's inequality,
\begin{align*}
	E[(\zeta_{t/d} - z)_+^{d-1}] &= E[(\zeta_{t/d} - z)_+^{d-1}1(\zeta_{t/d} > z)] \\
	&\le ( E[(\zeta_{t/d} - z)_+^d] )^{1 - 1/d}P(\zeta_{t/d} > z)^{1/d} = O(P(\zeta_{t/d} > z)^{1/d}), 
\end{align*}
when $\sup_t E[N_t] < \infty$. 
It turns out that, as in $d=1$, the order of $z$ necessary for $E[N_t] = O(1)$ is again at least $\sqrt{t\log t}$ (see the discussion at the end of Section \ref{pfmethods}). This means $P(\zeta_{t/d} > z) = o(1)$. Moreover, a consideration of the second order terms of $z$ shows 
$\sC_t(\{g_i\},z) = O( \gamma_d(t) P(\zeta_{t/d} > z)^{2/d} ) \to 0$
for any $d \ge 2$ (cf. Lemma \ref{errorrates}).

As alluded to before, the strong Rayleigh property of SSEP means that $\lim_{t\to\infty} E[N_t] = \lambda \in (0,\infty)$ and $\sS_t(\{g_i\},z) + \sC_t(\{g_i\},z) \to 0$ is sufficient for $N_t \Rightarrow \text{Poisson}(\lambda)$. When
 $g_i(u) = c_i u + r_i$, and for a fixed parameter $x\in\R$, 
 the scaling $z = z(t, x)$ and limit $\lambda = \lambda(x)$ can be found, which implies the Gumbel limit 
\begin{equation}\label{eq:linearGumbel}
\begin{aligned}
	&\lim_{t\to\infty} P\Big( X_t\sqrt{\frac{\log t}{t}} - \log t + \frac{d+1}{2d}\log \log t + \frac{\log 2\pi}{2d}
	\le x \Big)\\ 
&= \exp\Big( - \frac{(d-1)! \prod_i (2c_i)}{d^{d+{1/2}}} e^{-dx} \Big) 
\end{aligned} 
\end{equation}
for the maximal particle position $X_t$. 
 This is part of Theorem \ref{pyramidmain}, which more generally considers profiles where $g_i$ is a polynomial of arbitrary order. 
Additionally, we obtain the limit distributions of all order statistics of the process.

Notably, the right hand side of \eqref{eq:linearGumbel} does not depend on the intercepts $r_i$. An interesting consequence is that by varying these values, we may add an infinite number of particles to the initial system without changing the asymptotics (see also Remark \ref{polyremark} about omitting particles periodically from the initial profile). Generally, the limit of $E[N_t]$ (and thus the limit distribution of $X_t$) in all cases will depend only on the leading order behavior of $\prod_i g_i(u)$ as $u \to \infty$; see \eqref{eq:ENasympF} and the discussion in Remark \ref{mainremark}.

\subsection{Proof methods}\label{pfmethods}

 Using the {\em self-duality} of SSEP, along with the related stirring construction mentioned previously, the asymptotic analysis of $N_t$ in any dimension reduces to that of a single continuous time random walk $\zeta_t$ on $\Z$ starting at $0$. Informally,
two Markov processes are {\em dual} if certain expectations of functionals of one can be written as analogous expectations of the other (we refer to \cite{JanKur14} for a general treatment). For our purposes, SSEP is {\em self-dual} because moments of occupation variables can be expressed as expectations in terms of particle positions. More precisely, when initially particles are placed according to $\eta \in \{0,1\}^{\Z^d}$, 
\begin{equation}\label{eq:selfduality}
	E[\eta_t(x_1) \cdots \eta_t(x_n) | \eta_0 = \eta] = E[\eta(Y_1(t))\cdots \eta(Y_n(t)) | Y(0) = (x_1, \ldots, x_n)], 
\end{equation}
where $Y(t) = (Y_1(t), \ldots, Y_n(t))$ gives the positions of particles in an $n$-particle exclusion system \cite[Theorem VIII.1.1]{LigBook05}. Note that when $n = 1$, the single-particle system $Y_1(t)$ behaves like a random walk. 

To bound $\sC_t(\{g_i\},z)$, we apply this duality to analyze the ``two-point'' functions 
\begin{equation}\label{eq:covduality}
	-\Cov(\eta_t(x), \eta_t(y)) = E[\eta_t(x)]E[\eta_t(y)] - E[\eta_t(x)\eta_t(y)]. 
\end{equation}
Along with the negative association properties of SSEP, the following Lemma \ref{covintrep} serves as our starting point. Its proof mirrors that of \cite[Lemma 3.4]{ConSet23}, and is postponed until Section \ref{helping-section}.

Here and throughout, let $e_j$, $1 \le j \le d$, denote the $j$th standard basis vector in $\Z^d$, namely the vector with $1$ as the $j$th component and $0$ in the remaining components. 

\begin{lemma}\label{covintrep} For any initial profile $\eta \in \{0,1\}^{\Z^d}$ and any $z \in \R$, 
\[
	\sC_t(\eta, z) \le \frac{1}{d} \sum_{j \in \Z^d} \sum_{k=1}^d \int_0^t \left( E_j[\eta(\zeta^{(d)}_s)] - E_{j + e_k}[\eta(\zeta^{(d)}_s)] \right)^2P_j(\zeta^{(d)}_{t-s} \cdot e_1 \ge z)^2\,ds, 
\]
where 
$\zeta^{(d)}_t$ is a continuous time simple random walk on $\Z^d$ with $P_j(\cdot) = P(\cdot | \zeta^{(d)}_0 = j)$. 
\end{lemma}

We remark that the bound in Lemma \ref{covintrep} holds in all $d\geq 1$, and was stated in Lemma 3.4 of \cite{ConSet23} when $d=1$.

In the proof of Proposition \ref{mainbound} given in Section \ref{mainboundpf}, we use the decomposition of the components of $\zeta_t$ as independent simple random walks on $\Z$ with jumps at rate $1/d$ to obtain bounds on $\sC_t(\{g_i\},z)$ in terms of the one-dimensional walk $\zeta_t = \zeta_t^{(1)}$. Under our assumptions on $\{g_i\}$, this results in bounds on $\sS_t(\{g_i\},z)$ and $\sC_t(\{g_i\},z)$ given in terms of quantities of the form 
$E[H(\zeta_t - z)1(\zeta_t > z)]$
for a specified nonnegative, nondecreasing function $H$ depending on $\{g_i\}$. Recalling \eqref{eq:ENasympF}, $E[N_t]$ is also expressed in this form with $H(v) = \int_0^v h(u)\,du$ for $h$ depending on $\{g_i\}$. In the Appendix (Section \ref{rwlemmas}), we list several results concerning such random walk functionals, whose proofs are straightforward.

In particular, Lemma \ref{rw2normal} states that if $H$ is differentiable with $H'$ bounded by a polynomial, $z = o(t^{2/3})$, and $X$ is standard Gaussian, 
\begin{equation}\label{eq:rwlim=gausslim}
	\lim_{t\to\infty} E[H(\zeta_t - z)1(\zeta_t > z)] = \lim_{t\to\infty} E[H(\sqrt{t} X - z)1(\sqrt{t}X > z)], 
\end{equation}
when the limits
exist. Reduction to asymptotics related to the Gaussian distribution is convenient, since in general evaluating the left hand side limit is nontrivial.

For general profiles, evaluating the level $z$ and the limit of $E[N_t]$ when $\sup_tE[N_t]<\infty$ is difficult, even when relations such as \eqref{eq:rwlim=gausslim} hold. 
Nevertheless, for the class of polynomial shape functions $\{g_i\}$ in Section \ref{gumbelsec}, we may evaluate
\begin{equation}\label{eq:formofEN}
	\lim_{t\to\infty} E[N_t] = \lim_{t\to\infty} E\Big[ \int_0^{(\zeta_{t/d} - z)_+}h(u)\,du \Big],
\end{equation}
using Gaussian approximations. 

The $z = o(t^{2/3})$ assumption for \eqref{eq:rwlim=gausslim} is so that certain random walk tail probabilities that show up in the proof lie in the regime of classical large deviation results. 
In the case that $\{g_i\}$ have polynomial growth, $h$ in \eqref{eq:formofEN} also has polynomial growth (and therefore so does $H(u) = \int_0^u h(v)\,dv$).  A consequence is that the growth of $z$ is limited.

Indeed, when $\sup_t E[N_t]<\infty$, we will observe necessarily $t^{-1/2}z\to \infty$ (Remark \ref{z/sqrtt}).
Further, 
suppose that $z \ge \sqrt{c t \log t}$ for sufficiently large $t$ and consider
$h(u) = u^{\beta-1}$ in \eqref{eq:formofEN}, corresponding to $\{g_i(u)\asymp u^{\alpha_i}\}$ and $\beta = 1+ \sum_{i=2}^d \alpha_i \ge 1$. We have 
by Cauchy-Schwarz, the Burkholder-Davis-Gundy bound $E[\zeta_t^{2\beta}] = O(t^\beta)$, and a standard moderate deviation estimate $P(\zeta_{t/d} > \sqrt{ct\log t})=O(t^{-cd/2})$ (see, e.g., \cite[Lemma A.3]{ConSet23}) that for large enough $t$,
\begin{align*}
	E[N_t] \asymp E[(\zeta_{t/d} - z)_+^\beta] &\le E[(\zeta_{t/d} - \sqrt{ct\log t})_+^\beta] \\ 
	&\le (E[\zeta_{t/d}^{2\beta}])^{1/2} P(\zeta_{t/d} > \sqrt{ct\log t})^{1/2}
	= O( t^{(1/2)(\beta - cd/2)}). 
\end{align*}
Thus $E[N_t] \to 0$ if $c > 2\beta/d$. This means that
a nontrivial weak limit for $N_t$ (and $X_t$) requires $z = O(\sqrt{t\log t})$. As we see from Theorem \ref{pyramidmain}, 
 to first order the appropriate scaling is $z \sim \sqrt{(\beta/d) t \log t}$.

\subsection{Open problems}
\label{open problems}

The thrust of the paper is to identify a ``step'' profile setting, sufficient conditions, and sufficient bounds to obtain Poisson limits of $N_t$ to match those present in the analogous independent system. 
In this context, 
we mention some questions of interest arising from our study, which probe further the structure of the level $z$ and the correlations present, for possible later development.  

\medskip

1. We identify in Section \ref{gumbelsec} the scaling $z$ so that $E[N_t]$ converges to a $\lambda>0$ for polynomial shapes $\{g_i\}$.  However, the class specified in our Conditions (A), (B), and (C) below allows shapes with growth between polynomial and exponential.
The appropriate scaling sequence $z$ so that $E[N_t]$ converges for these subexponential shapes is unknown. 
For example, can \eqref{eq:rwlim=gausslim} be leveraged? In this context, we wonder if the growth of $z$ can exceed $O(\sqrt{t\log t})$, the order for polynomial shapes. 

A back-of-the-envelope calculation suggests that $E[\exp(\zeta_t - z)1(\zeta_t > z)] = O(1)$ requires that $z$ is on the large deviation scale $z \asymp t$. A natural question, then, is whether the correct scaling for shapes between polynomial and exponential interpolates between $\sqrt{t\log t}$ and $t$, or whether the change is sharp. 
This question is relevant also for a system of independent particles, where the calculation of the limit $E[N_t]$ is exactly the same as in the symmetric exclusion processes.

\medskip
 
2. 
As mentioned earlier, the vanishing of the first term in \eqref{eq:introC} as $t \to \infty$ may be thought of as an additional condition (beyond boundedness of $E[N_t]$ and (A)--(C)) on the shape $\{g_i\}$ in $d=2,3$ so that the correlations vanish (see \eqref{eq:d=2,3cond} in Theorem \ref{meta}). 
In $d \ge 4$ (cf. Corollary \ref{Ccor}) or when $\{g_i\}$ are polynomials (cf. Lemma \ref{errorrates}), this condition is already implied by the others. 
It would be of interest to investigate the necessity of this additional requirement in $d=2,3$ in this light.

\section{Definitions and preliminaries}
\label{defn_section}

Let $\{\eta_t : t \ge 0\}$ be a symmetric, nearest-neighbor, translation-invariant exclusion process on $\Z^d$ for $d \ge 2$. Namely, $\eta_t$ is the process taking values in $\{0,1\}^{\Z^d}$ with Markov generator 
\[
	\sL_d f(\eta) = \frac{1}{2d} \sum_{x \in \Z^d} \sum_{|y-x| =1}  \eta(x)(1 - \eta(y)) \left( f(\eta^{x,y}) - f(\eta) \right), 
\]
for functions $f(\eta)$ that depend on $\eta(x)$ for finitely-many $x \in \Z^d$, and where $\eta^{x,y}(x) = \eta(y)$, $\eta^{x,y}(y) = \eta(x)$, and $\eta^{x,y}(u) = \eta(u)$ otherwise. 
For an initial condition $\eta \in \{0,1\}^{\Z^d}$, we denote by $\P_\eta$ the probability measure under which $\eta_0 = \eta$. The corresponding expectation operator is denoted $\E_\eta$. 

We introduce the quantities of interest as follows. 
For fixed $d\ge 2$, define the set 
\[
	\sP_t = \{k \in \Z : \eta_t(k, x_2, \ldots, x_d) = 1 \;\,\text{for some}\;\, (x_2, \ldots, x_d) \in \Z^{d-1}\}. 
\]
$\sP_t$ is the collection of points $k \in \Z$ for which a particle exists at time $t$ in the affine hyperplane orthogonal to $e_1$ and passing through the point $ke_1$ (recall that $e_j$ denotes the $j$th standard basis vector in $\Z^d$). 
Then, 
\[
	X_t = \max \sP_t
\]
gives the maximal particle position in the $e_1$ direction. The initial conditions we consider will ensure that this maximum exists and also that $\inf \sP_t = -\infty$ for all $t\ge0$, that is, that there are an infinite number of particles in the system.
So, we can define 
\begin{equation}\label{eq:orderstatdef}
	\cdots \le X_t^{(3)} \le X_t^{(2)} \le X_t^{(1)} \le X_t^{(0)} = X_t, 
\end{equation}
the order statistics of $\sP_t$.   
Note that $X_t$ may not correspond to a unique particle, and $X_t^{(m)} = X_t^{(l)}$ is possible for every pair $(m,l)$. To have $X_t < \infty$, we restrict ourselves to initial profiles $\eta \in \{0,1\}^{\Z^d}$ for which $\eta(x) = 0$ when $x$ lies outside the halfspace $\sH_d = \{x \in \Z^d : x_1 \le 0\}$. Throughout it is understood that for $x \in \Z^d$, $x_1, x_2, \ldots, x_d$ denote its component values. 

As mentioned previously, analysis of the processes $X_t^{(m)}$ is based on the related quantity 
\begin{equation}\label{eq:Ndefinition}
	N_t = N_t(z) = \underset{x_1 > z}{\sum_{x \in \Z^d}} \eta_t(x), 
\end{equation}
which counts the number of points in $\sP_t$ to the right of $z = z(t)$. (We will suppress $z$ from the notation and write $N_t$ for $N_t(z)$.) As mentioned in the Introduction, the relationship between this quantity and $X_t$ is given by 
$\{X_t \le z\} = \{N_t = 0\}$, 
i.e., the largest value in $\sP_t$ is at most $z$ if and only if no points in $\sP_t$ lie to the right of $z$. More generally, we have 
\begin{equation}\label{eq:orderstatNrel}
	\{X_t^{(m)} \le z\} = \{N_t \le m\}. 
\end{equation}

As discussed above, we construct initial profiles for which $N_t < \infty$.
For functions $g_2, \ldots, g_d : \R_+ \to \R_+$, where $\R_+ = [0, \infty)$, let 
\[
	\sR_{\{g_i\}_{i=2}^d} = \{x \in \sH_d : |x_i| \le g_i(-x_1)\;\,\mbox{for}\;\, i = 2, \ldots, d\}, 
\]
and define $\eta_{\{g_i\}} = \eta_{\{g_i\}_{i=2}^d} \in \{0,1\}^{\Z^d}$ by 
\[
	\eta_{\{g_i\}_{i=2}^d}(x) = 1(x \in \sR_{\{g_i\}_{i=2}^d}), \qquad x \in \Z^d. 
\]
For such an initial profile, we use the shorthand $\P_{\{g_i\}} = \P_{\eta_{\{g_i\}}}$ and $\E_{\{g_i\}} = \E_{\eta_{\{g_i\}}}$. Note that with these definitions, $X_0 = 0$ and more generally $X_0^{(m)} \in [-m,0]$ for all $m \ge 0$, $\P_{\{g_i\}}$-a.s.

Below we collect the various assumptions we will make on the shape functions. 

\begin{conditions} 
Let $g:\R_+\rightarrow\R_+$ be a function.
\begin{enumerate}[(A)]

\item\label{nondecrcond} $g$ is nondecreasing.

\item\label{covcond} $g$ is continuously differentiable and  
$$
\sup_{m \in \Delta(g)} \frac{|g(m) - g(m-1)|}{g'(m)} < \infty,$$
where $\Delta(g) = \{m \in \{1, 2, \ldots\} : g(m) \ne g(m-1)\}$.

\item\label{smallderiv} $g$ is continuously differentiable with 
$$\lim_{u \to \infty} \frac{g'(u)}{g(u)} = 0.$$
\end{enumerate}
\end{conditions}

\begin{remark}\rm
\begin{enumerate}[(a)]

\item
Note that the initial profiles are fashioned by starting with continuum $\{g_i\}$ and then determining a discrete subset $\sR_{\{g_i\}}$ of $\R^d$ in which to place particles. An alternative would to be to start from discrete $g_i : \Z_+ \to \Z_+$ satisfying discrete forms of the above conditions, and then use continuous interpolations in later arguments. 

\item
When Condition \ref{nondecrcond} is also satisfied, Condition \ref{covcond} may be restated as follows: There is a constant $C \in (0,\infty)$ so that for any integer $m \ge 1$, 
\begin{equation}\label{eq:AandBcond}
	g(m) - g(m-1) \le Cg'(m) \quad\text{whenever}\quad g(m) > g(m-1). 
\end{equation}
Thus we require that $g'(m) \ne 0$ when $g$ is non-constant on $[m-1,m]$. In particular, this is satisfied by any non-decreasing polynomial, convex function, or logarithmic function, among others. This is also satified vacuously by a constant function. 

Conditions \ref{nondecrcond}--\ref{smallderiv} as stated are convenient for proving our results. However, they may be weakened to each $g_i$ being sufficiently-well approximated at $\infty$ by a function that satisfies them (see Remark \ref{mainremark}). 
In this context, Condition \ref{covcond} does not meaningfully limit allowable initial profiles outside of the requirement that growth on $[m-1,m]$ is uniformly comparable to the derivative at the right endpoint $m$. That is, without changing the asymptotics of $g$ at $\infty$ we may assume that $g'(m) \ne 0$ whenever $g(m) \ne g(m-1)$. 

\item
Condition \ref{smallderiv} limits the growth of the function: A function $g$ satisfying Condition \ref{smallderiv} has $\lim_{u\to\infty} e^{-cu}g(u) = 0$ for any $c > 0$. See Remark \ref{condCexample} for further discussion on this assumption.

\end{enumerate}
\end{remark}

Our main results will be
stated in terms of the following functions related to an initial profile $\eta_{\{g_i\}}$. 

\begin{definition}\label{capitalfunctions} For functions $g_2, \ldots, g_d : \R_+ \to \R_+$ and a subset $A \subset \{2, \ldots, d\}$, define $G_A, \widehat G_A : \R_+ \to \R_+$ by 
\[
	G_A(u) = \prod_{i\in A} (2g_i(u)+1), \qquad \widehat G_A(u) = \int_0^u \sum_{i\in A} \prod_{l \in A\setminus\{i\}} (2g_l(v)+1)\,dv, 
\]
with the convention that $G_\varnothing \equiv 1$, $\widehat G_\varnothing \equiv 0$, and $\widehat G_{\{i\}}(u) = u$. Moreover, denote 
\[
	G(u) = G_{\{2, \ldots, d\}}(u), \qquad \widehat G(u) = \widehat G_{\{2, \ldots, d\}}(u). 
\]
Note that when $d = 2$, $\widehat G(u) = \widehat G_{\{2\}}(u) = u$.
\end{definition}

\begin{remark}
\label{G-rmk}
\rm

\begin{enumerate}[(a)]

\item 
When $\{g_i\}_{i\in A}$ all satisfy both Conditions \ref{nondecrcond} and \ref{covcond}, then
$G_A$ also satisfies those conditions. Similarly, when $\{g_i\}_{i\in A}$ all satisfy both Conditions \ref{nondecrcond} and \ref{smallderiv}, then $G_A$ does as well. Thus when all $g_2, \ldots, g_d$ satisfy Conditions \ref{nondecrcond}, \ref{covcond}, and \ref{smallderiv} (as is the assumption of our main result, Theorem \ref{meta}), so does $G = G_{\{2, \ldots, d\}}$. In particular, this implies that  there is a universal constant $C>0$ so that $G(m+1) \le CG(m)$ for any integer $m \ge 0$. This last claim is shown in Lemma \ref{Gproperty}. 

\item It will be useful to note that
\begin{equation}\label{eq:hatG'<G}
	\widehat G_A'(u) \le (d-1) G_A(u) \qquad\text{and}\qquad \widehat G_A(u) \le (d-1) \int_0^u G_A(v)\,dv, 
\end{equation}
where $\widehat G_A'(u)$ exists for all $u$ when $\{g_i\}_{i \in A}$ are continuous. 

\item $G$ can be seen to represent the cross-sectional volume of $\sR_{\{g_i\}}$ for a fixed coordinate in the $-e_1$ direction. 
On the other hand,  $\widehat G$ arises from a consideration of the difference of cross-sectional volumes shifted on the axis spanned by $e_i$, $i > 1$ (more technically seen in the proof of Lemma \ref{positivealphabound}).

\end{enumerate}
\end{remark}

\subsection{The stirring process}\label{stir}

The process $\{\eta_t\}$ has an alternate construction using the ``stirring'' variables 
$\{\xi_x(t) :  t \ge 0,
	\, x \in \sH_d\}$.
For an initial profile $\eta$, $\xi_x(0) = x$ for $\eta(x) = 1$, and each $\xi_x(t)$ corresponds initially to a particle position. $\xi_x(t)$ then evolves like a continous time simple random walk on $\Z^d$ with a clock independent of the others, subject to an exclusion rule with ``swapping.'' That is, if $\xi_x(t) = u$ and $\xi_y(t) = v$ with $|u-v| = 1$ and the process $\xi_x(t)$ attempts a jump to site $v$ at a time $t + \eps$, then $\xi_x(t+\eps) = v$ and $\xi_y(t + \eps) = u$. While the jump by the particle at $u$ at time $t$ is suppressed, the stirring variables switch positions, and thus switch corresponding particles. (For further details on this construction see \cite[pg. 399]{LigBook05}.)

Thus the collection 
$\{\xi_x(t) : \eta_0(x) = 1\}$,
gives all the particle positions at time $t$, but each individual $\xi_x(t)$ does not track a single particle. Moreover, $\sP_t = \{\xi_x(t) \cdot e_1 : \eta_0(x) = 1\}$. 
This also means that marginally, each $\xi_x(t)$ evolves like a simple random walk. That is, 
\[
	\P_\eta(\xi_x(t) \in \cdot) = P_x(\zeta_t^{(d)} \in \cdot), \qquad \eta(x) = 1, 
\]
where we recall that $\{\zeta_t^{(d)}\}$ denotes a continous time simple random walk on $\Z^d$ and $P_x$ is the probability measure under which $\zeta_0 = x$. 

In our analysis, we are able to reduce many quantities in terms of a one-dimensional simple random walk, so we introduce the simplified notation $\zeta_t = \zeta_t^{(1)}$ for the simple random walk on $\Z$, with $P_k(\zeta_0 = k) = 1$. We will frequently use the fact that, if $\zeta_{1,t}, \zeta_{2,t}, \ldots, \zeta_{d,t}$ denote independent copies of $\zeta_t$, then 
\begin{equation}\label{eq:rwindependentdecomp}
	\zeta_t^{(d)} \overset{\text{Law}}{=} ( \zeta_{1,t/d}, \zeta_{2,t/d}, \ldots, \zeta_{d,t/d} ). 
\end{equation}
Thus also each component of the stirring variables obeys 
\begin{equation}\label{eq:onedstirring}
	\P_\eta(\xi_x(t) \cdot e_i \in \cdot) = P_{x_i}(\zeta_{t/d} \in \cdot), \qquad \eta(x) = 1, \qquad i = 1, \ldots, d. 
\end{equation}

For a sequence $z = z(t)$, the quantity $N_t = N_t(z)$ in \eqref{eq:Ndefinition} may be alternately expressed in terms of the stirring variables as 
\begin{equation}\label{eq:Nintermsofstirrings}
	N_t = \sum_{x \in \sH_d} \eta_0(x) 1(\xi_x(t) \cdot e_1 > z) = \sum_{x \in \sR_{\{g_i\}}} 1(\xi_x(t) \cdot e_1 > z), 
\end{equation}
where the second equality is true $\P_{\{g_i\}}$-surely.
Then, 
\begin{equation}\label{eq:generalNexpression}
\begin{aligned}
	\E_{\{g_i\}}[N_t] &= \sum_{x \in \sR_{\{g_i\}}} P_{x_1}(\zeta_{t/d} > z)
	= \sum_{x_1 \le 0} \sum_{|x_2| \le g_2(-x_1)} \cdots \sum_{|x_d| \le g_d(-x_1)} P_{x_1}(\zeta_{t/d} > z) \\
	&= \sum_{j \ge 0} \Big( \prod_{i=2}^d (2\lfloor g_i(j) \rfloor + 1) \Big) P_{0}(\zeta_{t/d} > z + j).  
\end{aligned}
\end{equation}
Here, we used \eqref{eq:onedstirring} and the fact that $P_k(\zeta_t \in \cdot) = P_0(\zeta_t + k \in \cdot)$ for any $k \in \Z$. 

\begin{remark}\label{z/sqrtt} \upshape 
From \eqref{eq:generalNexpression}, for $\sup_{t \ge 0} \E_{\{g_i\}}[N_t] < \infty$ it is necessary that $t^{-1/2}z \to \infty$, by the central limit theorem.
\end{remark}

The expression \eqref{eq:generalNexpression} has a more convenient asymptotic form when the $\{g_i\}$ are sufficiently regular and have subexponential growth, given in the following lemma. (Its proof is contained in Section \ref{helping-section}.)
Recall the functions $G_A$ and $\widehat G_A$ for $A \subset \{2, \ldots, d\}$ from Definition \ref{capitalfunctions}. 

\begin{lemma}\label{meanasymp} Suppose that $g_2, \ldots, g_d$ are continuous and satisfy Condition \ref{nondecrcond}. Let $U = \{ 2 \le i \le d : \sup_{u \ge 0} g_i(u) = \infty\}$ and $B = \{2, \ldots, d\} \setminus U$ denote the indices of the unbounded and bounded functions, respectively. For $i \in B$, denote
$L_i = \lim_{u \to \infty} (2 \lfloor g_i(u) \rfloor +1)$.

Then, there is $C > 0$ so that for all $t \ge 0$, 
\begin{equation}\label{eq:ENexpressionbound}
\begin{aligned}
	&\bigg| \E_{\{g_i\}}[N_t]  - \Big( \prod_{i \in B} L_i \Big) E_0\Big[ \int_0^{(\zeta_{t/d} - z)_+} G_U(u)\,du \Big] \bigg| \\
	&\le C (E_0[G_U(\zeta_{t/d} - z)1(\zeta_{t/d} > z) ] + E_0[ \widehat G_U(\zeta_{t/d} - z)1(\zeta_{t/d} > z) ] ), 
\end{aligned}
\end{equation}
with the convention $\prod_{i \in \varnothing} L_i = 1$. 

Moreover, if in addition $\{g_i\}_{i \in U}$ satisfy Condition \ref{smallderiv} and 
\[
	\sup_{t \ge 0} E_0\Big[ \int_0^{(\zeta_{t/d} - z)_+} G_U(u)\,du \Big] < \infty, 
\]
then $\lim_{t\to\infty}  (E_0[G_U(\zeta_{t/d} - z)1(\zeta_{t/d} > z) ] + E_0[ \widehat G_U(\zeta_{t/d} - z)1(\zeta_{t/d} > z) ] ) = 0$. 
\end{lemma}

We now enumerate two important cases of the lemma: 
\medskip

\noindent
{\it Case 1.} Suppose all $\{g_i\}$ are bounded, continuous, and nondecreasing.  Let $c_i = \sup_u \lfloor g_i(u) \rfloor$. Then $G_U \equiv 1$ and $\widehat G_U \equiv 0$, and Lemma \ref{meanasymp} says that 
\[
	\E_{\{g_i\}}[N_t] = \Big( \prod_{i=2}^d (2c_i +1) \Big) E_0[(\zeta_{t/d} - z)_+] + O(P_0(\zeta_{t/d} > z)), \qquad t\to\infty, 
\] 
where we must have $P_0(\zeta_{t/d} > z) = o(1)$ for $\lim_{t\to\infty} \E_{\{g_i\}}[N_t]$ to exist, by Remark \ref{z/sqrtt}.
In particular, this includes the case of constant $\{g_i\}$, which corresponds to an initial profile where a ``strip'' of particles extends to $-\infty$ in the $e_1$ direction. When $g_i \equiv c < 1$ for all $i$, so that the system begins with a single line of particles at $(x, 0, \ldots, 0)$, $x \in \{0, -1, -2, \ldots\}$, the asymptotics $\E_{\{g_i\}}[N_t] \asymp E_0[(\zeta_{t/d} - z)_+]$ are the same as for the $d = 1$ case in \cite{ConSet23}, but with time scaled by a factor of $d^{-1}$. 

\medskip

\noindent
{\it Case 2.} When $\lim_{u \to \infty} g_i(u) = \infty$ for all $i$, then
Lemma \ref{meanasymp} shows
\begin{align*}
	\E_{\{g_i\}}[N_t] &= 
	E_0\Big[ \int_0^{(\zeta_{t/d} - z)_+} G(u)\,du \Big] \\
	&\quad +  O( E_0[G(\zeta_{t/d} - z)1(\zeta_{t/d} > z)]  + E_0[ \widehat G(\zeta_{t/d} - z)1(\zeta_{t/d} > z)] ), \qquad t\to\infty. 
\end{align*}
Moreover, when $\{g_i\}$ are subexponential in the sense of Condition \ref{smallderiv}, then
\[
	\lim_{t\to\infty} \E_{\{g_i\}}[N_t] = 
	\lim_{t\to\infty} E_0\Big[ \int_0^{(\zeta_{t/d} - z)_+} G(u)\,du \Big], 
\]
when $z$ is chosen so that the right hand limit exists (which would imply $t^{-1/2}z\rightarrow \infty$ by Remark \ref{z/sqrtt}).
See Lemma \ref{pyramidmean} for the evaluation when $g_i$ are polynomials.

\subsection{Negative association and the Rayleigh property}\label{Rayleigh}

It is well known that the positions of particles in the symmetric exclusion process are negatively correlated \cite[Proposition VIII.1.7]{LigBook05}. More generally, if $V_2(t)$ is the Markov semigroup for the process on $(\Z^d)^2$ that gives the locations of particles in a $2$-particle system and $U_2(t)$ is the semigroup for the motion of $2$ independent random walks with jumps at rate $1$, then 
\begin{equation}\label{eq:semigroupinequality}
	V_2(t) h(x) \le U_2(t) h(x), \qquad x \in (\Z^d)^2, 
\end{equation}
for any symmetric, positive definite function $h$. (A function $h$ of two variables is positive definite if $\sum_{j,k} c(j) h(j,k)c(k) \ge 0$ whenever $\sum_j|c(j)| < \infty$ and $\sum_j c(j) = 0$.) 

Applying \eqref{eq:semigroupinequality} to $h(j,k) = 1(\{j,k\} \subset A)$ for $A \subset \Z^d$ and recalling that the collection of stirring variables describes the positions of all particles in a system, we obtain the following correlation inequality, which is Lemma 1' in \cite{Arr83} and Lemma 4.12 in \cite[Ch. VIII]{LigBook05}. 

\begin{lemma}\label{stirringcor} For any $A \subset \Z^d$ and $x \ne y$, 
\[
	\P_\eta\left( \xi_x(t) \in A, \xi_y(t) \in A \right) \le P_x(\zeta_t^{(d)} \in A)P_y(\zeta_t^{(d)} \in A), 
\]
when $\eta(x) = \eta(y) = 1$. 
\end{lemma}

The occupation variables $\{\eta_t(x) : x \in \Z^d\}$ are also negatively correlated \cite[Lemma VII.1.36]{LigBook05}. In fact, they satisfy the {\em strong Rayleigh} property whenever the system is started from a product measure (including the deterministic profiles we consider here): For finite $A \subset \Z^d$, the generating function 
\[
	Q(s) = \E_\eta\Big[ \prod_{x \in A} s_x^{\eta_t(x)} \Big] , \qquad s = \{s_x\}_{x \in A} \in \R^A, 
\]
satisfies 
$Q(s) \partial^2_{s_xs_y} Q(s) \le \partial_{s_x} Q(s) \partial_{s_y} Q(s)$,
for any $x \ne y$, $s \in \R^A$, and $\eta \in \{0,1\}^{\Z^d}$ \cite{BorBraLig09}. When $A = \{x,y\}$ and $s_x = s_y = 1$, we obtain the well-known negative correlation of occupation variables:
\begin{equation}\label{eq:occnegcor}
	\P_\eta(\eta_t(x) = 1, \eta_t(y) = 1) \le \P_\eta(\eta_t(x) = 1)\P_\eta(\eta_t(y) = 1). 
\end{equation}
As shown in \cite{Lig09,Van10}, the strong Rayleigh property implies that for any $t$ and subset $A \subset \Z^d$, there exist independent Bernoulli random variables $\{\theta_t(x) : x \in A\}$ such that 
\[
	\sum_{x \in A} \eta_t(x) \overset{\text{Law}}{=} \sum_{x \in A} \theta_t(x). 
\]
This fact provides the criterion from \cite{Lig09} for Poisson convergence of sums of occupation variables in symmetric exclusion systems, given after the following definition. 

\begin{definition}\label{def-SCE}
\rm
For $\eta \in \{0,1\}^{\Z^d}$ and $A \subset \Z^d$ (which may depend on $t$), define 
\[
	\sS_t(\eta, A) = \sum_{x \in A} (\E_\eta[\eta_t(x)])^2, \qquad \sC_t(\eta, A) = - 2\sum_{\{x,y\} \subset A} \Cov_\eta(\eta_t(x), \eta_t(y)), 
\]
and $\sE_t(\eta, A) = \sS_t(\eta, A) + \sC_t(\eta, A)$. (Note that $\sC_t(\eta, A) \ge 0$ by \eqref{eq:occnegcor}.) 
We also use the following shorthand: $\sS_t(\eta, z)=\sS_t(\eta, ((z,\infty)\cap\Z) \times \Z^{d-1})$ and $\sS_t(\{g_i\}, A) = \sS_t(\eta_{\{g_i\}}, A)$, with $\sC_t(\eta,z)$, $\sE_t(\eta, z)$, $\sC_t(\{g_i\}, A)$, and $\sE_t(\{g_i\}, A)$ defined similarly.

\end{definition}

\begin{lemma}\label{LiggettPois} 
We have
\[
	\sum_{x \in A} \eta_t(x) \Rightarrow \mbox{Poisson}\,(\lambda), \qquad t \to \infty, 
\]
if and only if 
\[
	\lim_{t\to\infty} \sum_{x \in A} \E_\eta[\eta_t(x)] = \lambda \qquad\mbox{and} \qquad \lim_{t\to\infty} \sE_t(\eta, A) = 0. 
\]
\end{lemma}

This lemma follows from standard conditions for Poisson convergence of sums of independent Bernoulli random variables discussed in the Introduction, since 
\begin{align*}
	\sE_t(\eta, A) &= \E_\eta \Big[ \sum_{x \in A} \eta_t(x) \Big] - \Var_\eta \Big( \sum_{x \in A} \eta_t(x) \Big) \\
	&= E\Big[ \sum_{x \in A} \theta_t(x) \Big] - \Var \Big( \sum_{x \in A} \theta_t(x) \Big) = \sum_{x \in A} (E[\theta_t(x)])^2, 
\end{align*}
for each $t$. 
Thus, the necessary and sufficient conditions for $N_t$ as given in \eqref{eq:Ndefinition} to converge in distribution to $\mbox{Poisson}(\lambda)$ can be stated identically for the exclusion system and the corresponding system of independent particles, namely 
$\lim_{t\to\infty} \E_{\eta}[N_t] = \lim_{t\to\infty} \Var_{\eta}(N_t) = \lambda$.

It will be useful to note the following monotonicity of $\sE_t$. Its proof is the same as Lemma 3.3 of \cite{ConSet23} using Lemma \ref{stirringcor} and is omitted. 

\begin{lemma}\label{monotoneerror} If $\eta(x) \le \tilde \eta(x)$ for every $x \in \Z^d$, then $\sE_t(\eta, A) \le \sE_t(\tilde \eta, A)$. 
\end{lemma}

\subsection{Notation and conventions}

In addition to notation already introduced, we have the following. 

We use $x(u) = O(y(u))$ as $u \to \infty$ to mean $x(u) \le C y(u)$ for some $C$ and sufficiently large $u$, and $x(u) \asymp y(u)$ as $u \to\infty$ means $x(u) = O(y(u))$ and $y(u) = O(x(u))$. $x(u) \sim y(u)$ as $u \to \infty$ denotes $\lim_{u\to\infty} x(u)/y(u) = 1$, and $x(u) = o(y(u))$ as $u \to \infty$ denotes $\lim_{u \to \infty} x(u)/y(u) = 0$. When clear, we will often omit ``$u \to \infty$'' from the notation. 

Other then where specified, the value $C$ is treated as a universal constant, and may correspond to different values in the same proof. The same is true for $C'$, $C''$, etc. 

The random variable $X$ will always refer to a standard Gaussian random variable defined on some space with probability measure $P$. The standard Gaussian density function is denoted by $\varphi$ and its cumulative distribution function is $\Phi$.

\section{Main results}\label{mainresults}

The following theorem presents our first main result. Fix a sequence $z = z(t)$ and define $N_t$ in terms of $z$ as in \eqref{eq:Ndefinition}. 
Recall the definition of the function $G$ from Definition \ref{capitalfunctions}. Also recall the quantities $\sS_t(\{g_i\}, z)$, $\sC_t(\{g_i\},z)$, and $\sE_t(\{g_i\},z)$ from Definition \ref{def-SCE}.

\begin{theorem}\label{meta} Suppose that $g_2, \ldots, g_d$ satisfy Conditions \ref{nondecrcond}, \ref{covcond}, and \ref{smallderiv}
and that 
\[
	\sup_{t \ge 0} \E_{\{g_i\}}[N_t] < \infty. 
\]
If $d \ge 4$, then 
\begin{equation}\label{eq:errorlimit0}
	\lim_{t\to\infty} \sE_t(\{g_i\},z) = \lim_{t\to\infty} (\sS_t(\{g_i\},z) + \sC_t(\{g_i\},z))= 0.
\end{equation}
If $d \in \{2,3\}$ and in addition 
\begin{equation}\label{eq:d=2,3cond}
	E_0[ G(\zeta_{t/d} - z)1(\zeta_{t/d} > z)] = \begin{cases} o\left((\log t)^{-1/2}\right) &\;\text{if }\, d = 3 \\[0.5em]
	o\left(t^{-1/4}\right) &\;\text{if }\, d = 2, \end{cases} \qquad t \to \infty, 
\end{equation}
then \eqref{eq:errorlimit0} holds. 
\end{theorem}

\begin{remark}\label{mainremark} \upshape
The result still holds when assumptions on $\{g_i\}$ are weakened. In particular, suppose that there are functions $f_2, \ldots, f_d$ that satisfy Conditions \ref{nondecrcond}, \ref{covcond}, and \ref{smallderiv} and $g_i(u) \le f_i(u)$ for all $i$. Then from Lemma \ref{monotoneerror}, 
\begin{equation}\label{eq:fleg}
	\sE_t(\{g_i\}, z) \le \sE_t(\{f_i\},z). 
\end{equation}
As long as
\begin{align}
\label{rm3.1}
	\sup_{t \ge t_0} \frac{\E_{\{f_i\}}[N_t]}{\E_{\{g_i\}}[N_t]} < \infty, 
\end{align}
for some $t_0 > 0$, then one can conclude \eqref{eq:errorlimit0} from first applying the theorem using $\{f_i\}$. In practice, this means that limiting behavior of $N_t$ depends on the shape functions $\{g_i\}$ only through the leading order asymptotics of $G(u)$ as $u \to \infty$ (see Remark \ref{polyremark} \ref{polyfirstorder}).

In \eqref{rm3.1}, the level $z$ used to define $N_t$ in both expectations is the same.
Also note that $g_i \le f_i$ for all $i$ implies $\E_{\{g_i\}}[N_t] \le \E_{\{f_i\}}[N_t]$ for all $t$. In particular, if $\{g_i\}$ are constrained only to a finite number of vertices in $\Z^d$ (that is, there is a finite number of particles in the system), then since $t^{-1/2}z\rightarrow \infty$ for $\E_{\{f_i\}}[N_t]$ to be finite by Remark \ref{z/sqrtt}, we would have $\E_{\{g_i\}}[N_t]\rightarrow 0$, a violation of \eqref{rm3.1}.
\end{remark}

From Lemma \ref{LiggettPois}, we immediately obtain the following corollary, to which the previous remark also applies.  

\begin{corollary}\label{poissonlimittheorem} Suppose that $g_2, \ldots, g_d$ satisfy Conditions \ref{nondecrcond}--\ref{smallderiv} and that 
\[
	\lim_{t\to\infty} \E_{\{g_i\}}[N_t] = \lambda < \infty. 
\]
If $d \ge 4$, then $N_t \Rightarrow \text{Poisson}(\lambda)$. If $d \in \{2, 3\}$ and \eqref{eq:d=2,3cond} holds, then $N_t \Rightarrow \text{Poisson}(\lambda)$. 
\end{corollary}

\begin{remark}\rm
\label{independence remark}

\begin{enumerate}[(a)]

\item As mentioned in the Introduction, Corollary \ref{poissonlimittheorem} shows for a large class of shapes $\{g_i\}$ that the exclusion and independent particle systems have the same behavior with respect to the Poisson convergence, needing only that $E_{\{g_i\}}[N_t]$ has a limit.  Indeed, in $d\geq 4$, all shapes $\{g_i\}$ considered are allowed.  While in $d=2,3$, the sufficient condition \eqref{eq:d=2,3cond}, since the level $z$ so that $E_{\{g_i\}}[N_t]$ converges would be the same as for independent particles, is a geometric one on the shape $\{g_i\}$, valid in particular for natural shapes such as those with polynomial growth, considered later in Theorem \ref{pyramidmain}.  
Exploring the necessity of \eqref{eq:d=2,3cond} 
is an open problem, mentioned in Subsection \ref{open problems}. 

\item Recalling \eqref{eq:orderstatNrel}, in particular that $\{X_t \le z\} = \{N_t = 0\}$, Corollary \ref{poissonlimittheorem} implies that $P(X_t \le z) \to e^{-\lambda}$. Thus when the scaling is of the form $z = b_t(x + a_t)$ for $x \in \R$, we may conclude that the limiting cumulative distribution function of $b_t^{-1}X_t - a_t$ at $x$ is $e^{-\lambda}$, where $\lambda = \lambda(x)$. In this manner we derive a Gumbel limiting distribution for $X_t$ for the examples in Section \ref{gumbelsec}, where $\lambda(x)$ is of the form $Me^{-\beta x}$. 

\item We observe in passing that if, in the context of this corollary, $z \to \infty$ is chosen so that the limit $\lambda =0$, then $N_t$ converges weakly to the point mass $\delta_0$.

\end{enumerate}
\end{remark}

Theorem \ref{meta} is proved using the next two propositions and the following corollary. Recall the functions $G$ and $\widehat G$ from Definition \ref{capitalfunctions}.

\begin{prop}\label{ssbound} Suppose that $g_2, \ldots, g_d$ satisfy Condition \ref{nondecrcond}. 
Then for any $t \ge 0$ and $z \in \R$, 
\[
	\sS_t(\{g_i\}, z) \le \E_{\{g_i\}}[N_t]E_0[G(\zeta_{t/d} - z)1(\zeta_{t/d} > z)]. 
\]
\end{prop}

We now state a bound on $\sC_t(\{g_i\},z)$ when $\{g_i\}$ also satisfy Condition \ref{covcond}. 
It is given in terms of the quantity $\gamma_d(t)$ defined in \eqref{gamma-d}.

\begin{prop}\label{mainbound} 
Suppose $g_2, \ldots, g_d$ all satisfy Conditions \ref{nondecrcond} and \ref{covcond}. Then there is $C = C(\{g_i\}) \in (0,\infty) $ so that, for all sufficiently large $t$ and any $z \in \R$, 
\begin{align*}
	\sC_t(\{g_i\}, z) 
	&\le C\big( \gamma_d(t)(E_0[ G(\zeta_{t/d} - z) 1(\zeta_{t/d} \ge z)] + E_0[ G'(\zeta_{t/d} - z) 1(\zeta_{t/d} \ge z)] )^2  \\
	&\qquad\quad +  \gamma_{d+1}(t) E_0[ G(\zeta_{t/d} - z) 1(\zeta_{t/d} \ge z)] E_0[ \widehat G(\zeta_{t/d} - z) 1(\zeta_{t/d} \ge z)] \big). 
\end{align*}
\end{prop}

Furthermore, the estimate in Proposition \ref{mainbound} simplifies when $\E_{\{g_i\}}[N_t] = O(1)$ and $\{g_i\}$ also satisfy Condition \ref{smallderiv}, as stated in the following corollary. 
\begin{corollary}\label{Ccor} Suppose $g_2, \ldots, g_d$ all satisfy Conditions \ref{nondecrcond}, \ref{covcond}, and \ref{smallderiv}. If a scaling sequence $z$ is such that $\sup_{t\ge0}\E_{\{g_i\}}[N_t] < \infty$, then there is $C = C(\{g_i\}) \in (0,\infty) $ so that, for all sufficiently large $t$, 
\begin{align*}
	\sC_t(\{g_i\},z) &\le C \big( \gamma_d(t)( E_0[ G(\zeta_{t/d} - z) 1(\zeta_{t/d} \ge z)])^2 + \gamma_{d+1}(t)  E_0[ G(\zeta_{t/d} - z) 1(\zeta_{t/d} \ge z)] \big). 
\end{align*}
Moreover, 
\begin{equation}\label{eq:Gto0}
	\lim_{t\to\infty} E_0[ G(\zeta_{t/d} - z) 1(\zeta_{t/d} \ge z)]  = 0. 
\end{equation}
\end{corollary}

\begin{remark}\label{condCexample}
\upshape 
\begin{enumerate}[(a)]

\item
We observe that Condition \ref{smallderiv} is not needed to prove Propositions \ref{ssbound} and \ref{mainbound}. However, the bounds therein are only applicable in the proof of Theorem \ref{meta} under the additional third assumption Condition \ref{smallderiv}, since otherwise they may not be small. For example, $G(u) = e^u$ is excluded by Condition \ref{smallderiv} but covered by Conditions \ref{nondecrcond} and \ref{covcond}. Since $\int_0^u G(v)\,dv \asymp G(u) $, from Lemma \ref{meanasymp}, 
\[
	E_0[ G(\zeta_{t/d} - z) 1(\zeta_{t/d} \ge z)] \asymp \E_{\{g_i\}}[N_t], \qquad t\to\infty. 
\]
Then \eqref{eq:Gto0} is not implied by $\sup_t \E_{\{g_i\}}[N_t] < \infty$. 
Moreover, if $d \in \{2,3\}$ and $\E_{\{g_i\}}[N_t]$ converges in $(0, \infty)$, then 	
$\gamma_d(t) (E_0[ G(\zeta_{t/d} - z) 1(\zeta_{t/d} \ge z)])^2 \to \infty$.
In $d \ge 4$, we may conclude that $\sup_t \sE_t(\{g_i\}, z) < \infty$ but not that $\sE_t(\{g_i\}, z) \to 0$. 

\item Since $\eps^2 = O(\eps)$ as $\eps \to 0$ and $\gamma_{d+1}(t) \le \sqrt{\gamma_d(t)}$, when \eqref{eq:d=2,3cond} holds 
the bound for $\sC_t(\{g_i\},z)$ in Corollary \ref{Ccor} is seen to be of order $\sqrt{\gamma_d(t)} E_0[G(\zeta_{t/d} - z)1(\zeta_{t/d} \ge z)]$ for all $d \ge 2$, which gives Theorem \ref{meta}. We previously refered to \eqref{eq:d=2,3cond} as an additional geometric condition on $z$ and the shape functions $\{g_i\}$. Further discussion can be found in Subsection \ref{open problems}. 

\end{enumerate}

\end{remark}

\section{A Gumbel limit theorem}\label{gumbelsec}

Here we consider a class of polynomial shapes
where, for each $i = 2, \ldots, d$, 
\begin{equation}\label{eq:polyfdef}
	g_i(u) = c_i u^{\alpha_i} + r_i, \qquad c_i \in (0, \infty), \; \alpha_i, r_i \in [0,\infty). 
\end{equation}
Note that in the case $\alpha_i = 0$, we simply have a constant function $g_i(u) = c_i + r_i$. Let $\beta = 1 + \sum_{i=2}^d \alpha_i$
and 
\begin{equation}\label{eq:polyscaling}
	a_t = \log \Big( \frac{t}{(2\pi)^{1/(2\beta)}(\log t)^{(\beta+1)/(2\beta)}} \Big), \qquad b_t = \Big(\frac{\beta t}{d\log t}\Big)^{1/2}. 
\end{equation}

For the main result of this section, recall the definitions of the order statistics $X_t^{(m)}$, $m \ge 0$, of the process $\eta_t$ in \eqref{eq:orderstatdef}. In particular, $X_t = X_t^{(0)}$ gives the maximal particle position in the $e_1$ direction.

\begin{theorem}\label{pyramidmain} 
For $\{g_i\}$ in \eqref{eq:polyfdef}, $a_t$ and $b_t$ in \eqref{eq:polyscaling}, each $m = 0, 1, 2, \ldots$, and any $x \in \R$, 
\begin{equation}\label{eq:GumbelCDF}
	\lim_{t\to\infty}  \P_{\{g_i\}}\Big(\frac{X_t^{(m)}}{b_t} - a_t \le x \Big) = \sum_{k=0}^m \frac{(M e^{-\beta x})^k}{k!} \exp \left( - M e^{-\beta x} \right), 
\end{equation}
where 
\begin{equation}\label{eq:Mbeta}
	M =  \frac{\Gamma(\beta) }{d^{\beta/2}\beta^{(\beta+1)/2}}\Big( \prod_{\alpha_i = 0} (2 \lfloor c_i + r_i \rfloor + 1) \Big) \Big(\prod_{\alpha_i > 0} 2c_i \Big). 
\end{equation}
\end{theorem}

\begin{remark}\label{polyremark} \upshape 
\begin{enumerate}[(a)]

\item \label{polyfirstorder} The form of \eqref{eq:polyfdef} here is chosen for ease of exposition. 
From the discussion in Remark \ref{mainremark}, Theorem \ref{pyramidmain} holds more generally when $g_i(u)$ is a polynomial with leading order term $c_iu^{\alpha_i}$. 

An alternate way to see this is as follows. Let $g_i$ be as in \eqref{eq:polyfdef}, and let $f_i$ be the positive part (to preserve nonnegativity) of an arbitrary polynomial with leading term $c_iu^{\alpha_i}$. Define $(N_t^{\{g_i\}}, N_t^{\{f_i\}})$ with the same scaling $z$ as a coupling of $\P_{\{g_i\}}(N_t \in \cdot)$ and $\P_{\{f_i\}}(N_t \in \cdot)$ on a space with probability measure $\tilde \P$. Then from \eqref{eq:Nintermsofstirrings}, 
\[
	\tilde \E | N_t^{\{g_i\}} - N_t^{\{f_i\}} | = O( E_0[(\zeta_{t/d} - z)^\theta] ), \qquad t \to \infty, 
\]
for some $\theta < \beta$. By the argument in \eqref{4-star} below, $E_0[(\zeta_{t/d} - z)^\theta] \to 0$, and thus $\P_{\{f_i\}}(N_t \in \cdot)$ converges to a certain Poisson distribution if and only if $\P_{\{g_i\}}(N_t \in \cdot)$ does. From our proof of Theorem \ref{pyramidmain} given below, this implies that \eqref{eq:GumbelCDF} also holds with $\{g_i\}$ replaced by $\{f_i\}$.

\item When $m = 0$, the limit distribution function in \eqref{eq:GumbelCDF} is that of a Gumbel random variable with mean 
$\mu = \gamma/\beta + \log M$,
where $\gamma \approx 0.5772$ is Euler's constant, and variance 
$\sigma^2 = \pi^2/(6\beta^2)$. 
Thus, for large $t$, 
\[
	\E_{\{g_i\}}[X_t] \approx b_t(\mu + a_t) \approx  \Big(\frac{\beta}{d}t\log t\Big)^{1/2}
\quad\text{and}\quad
	\Var_{\{g_i\}}(X_t) \approx b_t^2\sigma^2 = \frac{\pi^2}{6\beta} \frac{t}{\log t}. 
\]

\item Consider the case when $c_i = c > 0$ and $\alpha_i = 1$ for all $i$. Then, 
\[
	\P_{\{g_i\}}\Big(\frac{X_t}{b_t} - a_t \le x \Big) \to \exp\Big( - \frac{(d-1)!(2c)^{d-1}}{d^{d+1/2}}e^{-dx}\Big). 
\]
Stirling's approximation gives 
\[
	\frac{(d-1)!(2c)^{d-1}}{d^{d+1/2}}e^{-dx} = \frac{d!e^d}{d^{d+1/2}} \cdot \frac{(2ce^{-x - 1})^d}{2d} \sim \frac{\sqrt{2\pi}(2ce^{-x - 1})^d}{2d}, \qquad d \to \infty. 
\]
Thus for large $d$ the limiting distribution of $b_t^{-1}X_t-a_t$ is approximately a point mass at $\log (2c) - 1$.

\item As in \cite{ConSet23}, Theorem \ref{pyramidmain} can be extended to product measure initial conditions with a certain periodicity structure. Define $\nu$ on $\{0,1\}^{\Z^d}$ by 
\[
	\nu(\eta(x) = 1) = \rho_{x_1}, \qquad x = (x_1, \ldots, x_d) \in \sR_{\{g_i\}}, 
\]
and $\nu(\eta(x) = 1) = 0$ for $x \notin \sR_{\{g_i\}}$. Suppose that the collection $\{\rho_j : j \in \Z_-\} \subset [0,1]$ satisfies 
$\rho_{j-m} = \rho_j$ for some $m \ge 1$, 
and that $\rho_j > 0$ for at least one $j$. Then, 
\[
	\P_{\nu}\Big(\frac{X_t^{(m)}}{b_t} - a_t \le x \Big) \to \sum_{k=0}^m \frac{(\bar \rho M e^{-\beta x})^k}{k!} \exp\left( - \bar \rho M e^{-\beta x}\right), 
\]
where $\bar \rho = (1/m)\sum_{i=-m+1}^{0} \rho_i$. 

\end{enumerate}

\end{remark}

For the remainder of this section, fix $x \in \R$. Letting $z = b_t( x + a_t)$ for the scaling sequences in \eqref{eq:polyscaling}, we have
\begin{equation}\label{eq:pyramidz}
	z = \Big(\frac{\beta t}{d \log t}\Big)^{1/2} \Big(x - \frac{\log (2\pi)}{2\beta} + \log t - \frac{\beta+1}{2\beta} \log\log t \Big). 
\end{equation}
For notational convenience, we define $w = (t/d)^{-1/2}z$. Note for use throughout this section that $w \sim \sqrt{\beta \log t}$ and 
\[
	\frac{w^2}{2} = \frac{\beta}{2} \log t - \frac{\beta+1}{2} \log\log t + \beta x - \frac{1}{2}\log(2\pi) + o(1), \qquad t\to\infty. 
\]
Thus, recalling $\varphi$ denotes the standard Gaussian density function, 
\begin{equation}\label{eq:wasymp}
	\varphi(w) = \frac{(\log t)^{(\beta+1)/2}}{t^{\beta/2}}e^{-\beta x + o(1)}, \qquad t\to\infty. 
\end{equation}

To prove Theorem \ref{pyramidmain}, it suffices by \eqref{eq:orderstatNrel} to verify the conditions of Corollary \ref{poissonlimittheorem} and compute $\lim_{t\to\infty} \E_{\{g_i\}}[N_t]$. Note that $\{g_i\}$ all satisfy Conditions \ref{nondecrcond}, \ref{covcond}, and \ref{smallderiv}. 

Thus Lemma \ref{pyramidmean} below completes the proof of Theorem \ref{pyramidmain} in $d \ge 4$. However, the explicit initial profiles considered here allow us to compute quantitatively precise asymptotic rates of convergence for $\sS_t(\{g_i\}, z)$ and $\sC_t(\{g_i\}, z)$. These rates, given in Lemma \ref{errorrates}, confirm the additional requirement \eqref{eq:d=2,3cond} given in Theorem \ref{meta} for Theorem \ref{pyramidmain} to hold in $d = 2,3$. 

After Lemmas \ref{pyramidmean} and \ref{errorrates}, we provide the proof of Theorem \ref{pyramidmain}.

\begin{lemma}\label{pyramidmean} For $g_i$ in \eqref{eq:polyfdef}, $z$ in \eqref{eq:pyramidz}, and $M$ in \eqref{eq:Mbeta}, 
\[
	\lim_{t\to\infty} \E_{\{g_i\}}[N_t] = M e^{-\beta x}. 
\]
\end{lemma}

\begin{proof} Let $U = \{i : \alpha_i > 0\}$ and $B = \{i : \alpha_i = 0\}$. Then $U$ and $B$ correspond to the notation of Lemma \ref{meanasymp}. 

If $U = \varnothing$. Then $G_U \equiv 1$ and $\widehat G_U \equiv 0$. 
Otherwise, because $\beta = 1 + \sum_{i \in U} \alpha_i$, 
\[
	G_U(u) = \prod_{i \in U} (2g_i(u) + 1) = \prod_{i \in U} (2c_iu^{\alpha_i} + 2r_i + 1) = \Big( \prod_{i \in U} 2c_i \Big) u^{\beta - 1} + O(u^{\theta - 1}), 
\]
for some $\theta < \beta$. Then, 
$\int_0^u G_U(v)\,dv = (1/\beta) ( \prod_{i \in U} 2c_i ) u^\beta + O(u^\theta)$.
Moreover, 
$\widehat G_U(u) = \int_0^u \sum_{i \in U} \prod_{l \in U \setminus \{i\}} (2c_lv^{\alpha_l} + 1)\,dv = O(u^{\theta})$,
for some possibly different $\theta < \beta$. 

Thus, in general, there is $\theta < \beta$ so that  
\begin{equation}\label{eq:fromlem2.1}
	\E_{\{g_i\}}[N_t] = \frac{1}{\beta} \Big( \prod_{i \in B} (2\lfloor c_i + r_i \rfloor + 1) \Big) \Big( \prod_{i \in U} 2c_i \Big) E_0[(\zeta_{t/d} - z)_+^{\beta}] + O ( E_0[(\zeta_{t/d} - z)_+^{\theta}] ), 
\end{equation}
as $t \to \infty$,
from Lemma \ref{meanasymp}. 

Now, since $z = o(t^{2/3})$, Lemmas \ref{rw2normal} and \ref{lem:normal} along with \eqref{eq:wasymp} imply 
\begin{equation}\label{eq:betaasymp}
	E_0[ (\zeta_{t/d} - z)_+^\beta] 	\sim (t/d)^{\beta/2} E[(X - w)_+^\beta ] \sim \frac{\Gamma(\beta+1)(t/d)^{\beta/2}\varphi(w)}{w^{\beta+1}} \to \frac{\Gamma(\beta+1)}{d^{\beta/2}\beta^{(\beta+1)/2}}e^{-\beta x}.
\end{equation}
Because this last limit is finite and $0 \le \theta < \beta$, Lemma \ref{smallerfunction} implies that $E_0[(\zeta_{t/d} - z)_+^\theta] = o(1)$. Or,  Lemmas \ref{rw2normal} and \ref{lem:normal} may be used again to directly obtain 
\begin{equation}
\label{4-star}
\begin{aligned}
	E_0[(\zeta_{t/d} - z)_+^\theta] &\sim  (t/d)^{\theta/2} E[ (X - w)_+^{\theta}] \\
	&\sim \frac{\Gamma(\theta+1) (t/d)^{\theta/2}\varphi(w)}{w^{\theta+1}} \sim \frac{\Gamma(\theta+1)e^{-\beta x}}{d^{\theta/2} \beta^{(\theta+1)/2}} \Big( \frac{\log t}{t}\Big)^{(\beta - \theta)/2}  \to 0. 
\end{aligned}
\end{equation}
Combining this with \eqref{eq:fromlem2.1} and \eqref{eq:betaasymp} completes the proof. 
\end{proof}

The following lemma gives the rate of convergence of $\sE_t(\{g_i\},z) \to 0$ for $\{g_i\}$ in \eqref{eq:polyfdef} and $z$ in \eqref{eq:pyramidz}.

\begin{lemma}\label{errorrates} Let $\alpha^* = \min_{2 \le i \le d} \alpha_i$. There is a constant $C$ so that for all sufficiently large $t$, 
\begin{equation}\label{eq:polyss}
	\sS_t(\{g_i\},z) \le Ce^{-2\beta x} \sqrt{\frac{\log t}{t}}, 
\end{equation}
and 
\begin{equation}\label{eq:polycov}
	\sC_t(\{g_i\},z) \le C e^{-2\beta x} \Big( \gamma_d(t) \frac{\log t}{t} + \gamma_{d+1}(t) \Big( \frac{\log t}{t} \Big)^{(1 + \alpha^*)/2} \Big). 
\end{equation}
\end{lemma}

\begin{remark} \upshape The bound in \eqref{eq:polycov} gives an $\alpha^*$-dependent rate of convergence.
However, a more simply-stated upper bound for $\{g_i\}$ of the form \eqref{eq:polyfdef} holds if we bound $(\frac{\log t}{t})^{(1+\alpha^*)/2}\leq (\frac{\log t}{t})^{1/2}$:  For
 sufficiently large $t$, 
\begin{equation}\label{eq:worstrate}
	\sC_t(\{g_i\},z) \le Ce^{-2\beta x} \times \begin{cases} \dfrac{(\log t)^{3/2}}{\sqrt{t}} &\quad d = 2 \\[1em]  \sqrt{\dfrac{\log t}{t}} &\quad d \ge 3. \end{cases}
\end{equation}

\end{remark}

\begin{proof}[Proof of Lemma \ref{errorrates}] We apply Propositions \ref{ssbound} and \ref{mainbound}. 
From Lemma \ref{pyramidmean}, $\E_{\{g_i\}}[N_t] = O(e^{-\beta x})$. Moreover, from $G(u)\leq C(1+u^{\beta-1})$ and \eqref{4-star} with $\theta = 0$ and $\theta = \beta-1$, for sufficiently large $t$ we have 
\begin{align*}
	E_0[G(\zeta_{t/d} - z)1(\zeta_{t/d} \ge z)] &\le C \big( P_0(\zeta_{t/d} \ge z) + E_0[(\zeta_{t/d} - z)_+^{\beta-1}] \big) \\
	&\le C' e^{-\beta x} \Big( \Big(\frac{\log t}{t}\Big)^{\beta/2} + \Big(\frac{\log t}{t}\Big)^{1/2} \Big) \le C'' e^{-\beta x}\Big(\frac{\log t}{t}\Big)^{1/2}. 
\end{align*}
Thus \eqref{eq:polyss} follows from Proposition \ref{ssbound}. 

Next, since $G$ is a polynomial of order $\beta - 1 \ge 0$, there is $C > 0$ so that $G'(u) \le C G(u)$ for all $u \ge 0$. Again using $G(u) \le C(1 + u^{\beta-1})$, 
\begin{align*}
	\widehat G(u) &= \int_0^u \sum_{i=2}^d \prod_{l \ne i} (2c_l v^{\alpha_l} + 2r_l + 1)\,dv \\
	&\le \int_0^u G(v)\sum_{i=2}^d \frac{1}{2c_i v^{\alpha_i} + 1} \,dv \le C\Big(1 + \int_1^u v^{\beta - 1 - \alpha^*}\,dv\Big) \le C' (1 + u^{\beta - \alpha^*}).
\end{align*}
Hence from \eqref{4-star} with $\theta = 0$ and $\theta = \beta - \alpha^*$, 
\begin{align*}
	E_0[\widehat G(\zeta_{t/d} - z)1(\zeta_{t/d} \ge z)] &\le C \big( P_0(\zeta_{t/d} \ge z) + E_0[(\zeta_{t/d} - z)_+^{\beta - \alpha^*}] \big) \\
	&\le C' e^{-\beta x} \Big(\Big(\frac{\log t}{t}\Big)^{\beta/2} + \Big(\frac{\log t}{t}\Big)^{\alpha^*/2}\Big) \le C'' e^{-\beta x}  \Big(\frac{\log t}{t}\Big)^{\alpha^*/2}, 
\end{align*}
for sufficiently large $t$. Then from Proposition \ref{mainbound}, for sufficiently large $t$, 
\begin{align*}
	\sC_t(\{g_i\},z) &\le C \big( \gamma_d(t) (E_0[G(\zeta_{t/d} - z)1(\zeta_{t/d} \ge z)])^2 \\
	&\qquad\quad + \gamma_{d+1}E_0[G(\zeta_{t/d} - z)1(\zeta_{t/d} \ge z)] E_0[\widehat G(\zeta_{t/d} - z)1(\zeta_{t/d} \ge z)] \big) \\
	&\le C' \Big( \gamma_d(t) \Big( e^{-\beta x} \Big(\frac{\log t}{t}\Big)^{1/2} \Big)^2 + \gamma_{d+1}(t) \Big( e^{-\beta x} \Big(\frac{\log t}{t}\Big)^{1/2} \Big)\Big( e^{-\beta x} \Big(\frac{\log t}{t}\Big)^{\alpha^*/2} \Big) \Big) \\
	&= C' e^{-2\beta x} \Big( \gamma_d(t) \frac{\log t}{t} + \gamma_{d+1}(t) \Big( \frac{\log t}{t} \Big)^{(1 + \alpha^*)/2} \Big). \qedhere
\end{align*}
\end{proof}

We now prove Theorem \ref{pyramidmain}.

\begin{proof}[Proof of Theorem \ref{pyramidmain}] 
Via Corollary \ref{poissonlimittheorem}, we observe that Lemmas \ref{pyramidmean} and \ref{errorrates} imply 
$N_t \Rightarrow \text{Poisson}( M e^{-\beta x})$,
as $t \to \infty$ under $\P_{\{g_i\}}$, where $M$ is as in \eqref{eq:Mbeta}. Then from \eqref{eq:orderstatNrel}, recalling $z = b_t(x+a_t)$, as desired,
\[\P_{\{g_i\}}\Big( \frac{X_t^{(m)}}{b_t} - a_t \le x \Big) = \P_{\{g_i\}}(X_t^{(m)} \le z) = \P_{\{g_i\}}(N_t \le m) \to \sum_{k=0}^m \frac{(M e^{-\beta x})^k}{k!} e^{-M e^{-\beta x}}. \qedhere\]
\end{proof}

\section{Proofs of Lemmas \ref{covintrep} and \ref{meanasymp}}
\label{helping-section}

We begin with a proof of Lemma \ref{covintrep}. 

\begin{proof}[Proof of Lemma \ref{covintrep}] From \eqref{eq:covduality} and the self-duality property \eqref{eq:selfduality}, 
\[
	\sC_t(\eta,z) = 2\underset{x \cdot e_1,y \cdot e_1 > z}{\sum_{\{x,y\} \subset \Z^d}} \left[ U_2(t) - V_2(t) \right] \eta(x)\eta(y). 
\]
The function $\eta(x)\eta(y)$ is symmetric and positive definite, so $[ U_2(t) - V_2(t)] \eta(x)\eta(y) \ge 0$ by \eqref{eq:semigroupinequality}. Integrating by parts \cite[Ch. VIII]{LigBook05}, 
\begin{align*}
	\left[ U_2(t) - V_2(t) \right] \eta(x)\eta(y) &= \int_0^t V_2(t-s)\left[ \sU_2 - \sV_2 \right] U_2(s)\eta(x)\eta(y)\,ds,  
\end{align*}
where $\sU_2$ and $\sV_2$ are the Markov generators corresponding to $U_2(t)$ and $V_2(t)$, respectively: 
\begin{align*}
	\sU_2 f(x,y) &= \frac{1}{2d} \sum_{i=1}^d \Big[ \sum_{u \in \{x \pm e_i\}} \left( f(u,y) - f(x,y) \right) + \sum_{u \in \{y \pm e_i\}} \left( f(x,u) - f(x,y) \right) \Big], \qquad\text{and} \\
	\sV_2 f(x,y) &= \frac{1}{2d} \sum_{i=1}^d \Big[ \sum_{u \in \{x \pm e_i\}\setminus\{y\}} \left( f(u,y) - f(x,y) \right) + \sum_{u \in \{y \pm e_i\}\setminus\{x\}} \left( f(x,u) - f(x,y) \right) \Big], 
\end{align*}
for functions $f(x,y)$ in their domains. 
A computation with these generators yields 
\begin{align*}
	\left[ \sU_2 - \sV_2 \right] U_2(s)\eta(x)\eta(y) &= \frac{1(|x-y|=1)}{2d} U_2(s) \left( \eta(x)\eta(x) + \eta(y)\eta(y) - 2\eta(x)\eta(y) \right) \\
	&= \frac{1(|x-y|=1)}{2d} \left( E_x[\eta(\zeta^{(d)}_s)] - E_y[\eta(\zeta^{(d)}_s)] \right)^2. 
\end{align*}
Therefore, 
\begin{align*}
	\sC_t(\eta, z) &= \frac 1 d \underset{x \cdot e_1,y \cdot e_1 > z}{\sum_{\{x,y\} \subset \Z^d}} \int_0^t V_2(t-s) 1(|x-y| = 1) \left( E_x[\eta(\zeta^{(d)}_s)] - E_y[\eta(\zeta^{(d)}_s)] \right)^2\,ds \\
	&= \frac{1}{2d} \int_0^t \sum_{x \ne y} 1(x \cdot e_1, y \cdot e_1 > z) V_2(t-s) 1(|x-y| = 1) \left( E_x[\eta(\zeta^{(d)}_s)] - E_y[\eta(\zeta^{(d)}_s)] \right)^2\,ds \\
	&=  \frac{1}{2d} \int_0^t \sum_{x \ne y} 1(|x-y| = 1) \left( E_x[\eta(\zeta^{(d)}_s)] - E_y[\eta(\zeta^{(d)}_s)] \right)^2 V_2(t-s) 1(x \cdot e_1, y \cdot e_1 > z) \,ds, 
\end{align*}
where in the last equality we used the fact that $V_2(t-s)$ is a symmetric operator. Finally, $1(x\cdot e_1, y\cdot e_1 > z)$ is a symmetric, positive definite function, so by \eqref{eq:semigroupinequality}, 
\begin{align*}
	\sC_t(\eta, z) &= \frac{1}{2d} \sum_{x \in \Z^d}\sum_{i=1}^d \sum_{y \in \{x \pm e_i\}} \int_0^t \left( E_x[\eta(\zeta^{(d)}_s)] - E_y[\eta(\zeta^{(d)}_s)] \right)^2 V_2(t-s) 1(x \cdot e_1, y \cdot e_1 > z) \,ds \\
	&\le  \frac{1}{2d} \sum_{x \in \Z^d}\sum_{i=1}^d \sum_{y \in \{x \pm e_i\}} \int_0^t \left( E_x[\eta(\zeta^{(d)}_s)] - E_y[\eta(\zeta^{(d)}_s)] \right)^2 U_2(t-s) 1(x \cdot e_1, y \cdot e_1 > z) \,ds \\
	&= \frac 1 d \sum_{x \in \Z^d} \sum_{i=1}^d \int_0^t \left( E_x[\eta(\zeta^{(d)}_s)] - E_{x+e_i}[\eta(\zeta^{(d)}_s)] \right)^2 P_x(\zeta^{(d)}_{t-s} \cdot e_1 > z)P_{x+e_i}(\zeta^{(d)}_{t-s} \cdot e_1 > z)\,ds \\
	&\le  \frac 1 d \sum_{x \in \Z^d} \sum_{i=1}^d \int_0^t \left( E_x[\eta(\zeta^{(d)}_s)] - E_{x+e_i}[\eta(\zeta^{(d)}_s)] \right)^2 P_x(\zeta^{(d)}_{t-s} \cdot e_1 \ge z)^2\,ds. \qedhere
\end{align*}
\end{proof}

We conclude the section with the proof of Lemma \ref{meanasymp}.

\begin{proof}[Proof of Lemma \ref{meanasymp}] First note that, by \eqref{eq:generalNexpression}, monotonicity of $\{g_i\}$, and Lemma \ref{sum2exp}, 
\begin{align} \nonumber 
	\E_{\{g_i\}}[N_t] &\le  \Big( \prod_{i \in B} L_i \Big) \sum_{j \ge 0} G_U(j) P_0(\zeta_{t/d} > z + j) \\ \label{eq:ENupperbound}
	&\le \Big( \prod_{i \in B} L_i \Big) E_0\Big[ \int_0^{(\zeta_{t/d} - z)_+} G_U(u)\,du \Big] + \Big( \prod_{i \in B} L_i \Big) E_0[G_U(\zeta_{t/d} - z)1(\zeta_{t/d} > z)]. 
\end{align}
Now we prove a lower bound. 

When $g_i$ is bounded, then because it is nondecreasing and continuous, it is eventually larger than $\sup_u g_i(u) - 1$. 
This means there is $J \in \Z$ large enough so that 
\[
	\prod_{i \in B} ( 2\lfloor g_i(j) \rfloor + 1) = \prod_{i \in B} L_i \quad \text{for all} \quad j > J. 
\]
Thus, from \eqref{eq:generalNexpression}, 
\begin{equation}\label{eq:equalLeventually}
\begin{aligned}
	&\Big( \prod_{i \in B} L_i \Big) \sum_{j \ge 0} \Big( \prod_{i \in U} (2 \lfloor g_i(j) \rfloor + 1) \Big)P_0(\zeta_{t/d} > z + j) - \E_{\{g_i\}}[N_t] \\
	&\le \Big( \prod_{i \in B} L_i \Big) \sum_{j = 0}^{J} \Big( \prod_{i \in U} (2 \lfloor g_i(j) \rfloor + 1) \Big)P_0(\zeta_{t/d} > z + j) \le CP_0(\zeta_{t/d} > z), 
\end{aligned}
\end{equation}
where $C$ depends on $J$, $\{L_i : i \in B\}$, and $\{g_i(J) : i \in U\}$. 

Next, Since $ (2g_i(u) + 1) - (2 \lfloor g_i(u) \rfloor + 1) \le 2$ for all $i$ and $u$, 
\[
	\prod_{i \in U} (2g_i(u) + 1) - \prod_{i \in U} (2 \lfloor g_i(u) \rfloor + 1)  \le C \widehat G_U'(u).   
\]
Note that $\widehat G_U'(u)$ is nondecreasing. 
Thus, from Lemma \ref{sum2exp} and \eqref{eq:hatG'<G}, 
\begin{align} \nonumber
	& \sum_{j \ge 0} G_U(j)P_0(\zeta_{t/d} > z + j) - \sum_{j \ge 0} \Big( \prod_{i \in U} (2 \lfloor g_i(j) \rfloor + 1) \Big)P_0(\zeta_{t/d} > z + j) \\  \nonumber
	&= \sum_{j \ge 0} \Big(  \prod_{i \in U} (2g_i(j) + 1) - \prod_{i \in U} (2 \lfloor g_i(j) \rfloor + 1) \Big) P_0(\zeta_{t/d} > z + j) \\  \nonumber
	&\le C \sum_{j \ge 0} \widehat G_U'(j) P_0(\zeta_{t/d} > z + j) \\  \nonumber
	&\le C \big(E_0[\widehat G_U(\zeta_{t/d} - z)1(\zeta_{t/d} > z)] + E_0[\widehat G_U'(\zeta_{t/d} - z)1(\zeta_{t/d} > z)] \big) \\  \label{eq:EhatGUbound}
	&\le C(d-1)\big(E_0[\widehat G_U(\zeta_{t/d} - z)1(\zeta_{t/d} > z)] + E_0[G_U(\zeta_{t/d} - z)1(\zeta_{t/d} > z)] \big). 
\end{align}

Combining \eqref{eq:equalLeventually} and \eqref{eq:EhatGUbound}, and noting $P_0(\zeta_{t/d} > z) \le E_0[G_U(\zeta_{t/d} - z)1(\zeta_{t/d} > z) ]$, we obtain 
\begin{align*}
	\E_{\{g_i\}}[N_t] &\ge \Big( \prod_{i \in B} L_i \Big) \sum_{j \ge 0} \Big( \prod_{i \in U} (2 \lfloor g_i(j) \rfloor + 1) \Big)P_0(\zeta_{t/d} > z + j) - CP_0(\zeta_{t/d} > z) \\
	&\ge \Big( \prod_{i \in B} L_i \Big) \sum_{j \ge 0} G_U(j)P_0(\zeta_{t/d} > z + j) \\
	&\quad - C'\big(E_0[G_U(\zeta_{t/d} - z)1(\zeta_{t/d} > z) ] + E_0[ \widehat G_U(\zeta_{t/d} - z)1(\zeta_{t/d} > z) ] \big) - CP_0(\zeta_{t/d} > z) \\
	&\ge \Big( \prod_{i \in B} L_i \Big) \sum_{j \ge 0} G_U(j)P_0(\zeta_{t/d} > z + j) \\
	&\quad - C''\big(E_0[G_U(\zeta_{t/d} - z)1(\zeta_{t/d} > z) ] + E_0[ \widehat G_U(\zeta_{t/d} - z)1(\zeta_{t/d} > z) ] \big). 
\end{align*}
Another application of Lemma \ref{sum2exp} gives 
\[
	\sum_{j \ge 0} G_U(j)P_0(\zeta_{t/d} > z + j) \ge E_0\Big[ \int_0^{(\zeta_{t/d}-z)_+}G_U(u)\,du \Big] - E_0[ G_U(\zeta_{t/d} - z)1(\zeta_{t/d} > z) ], 
\]
and therefore
\begin{align*}
	\E_{\{g_i\}}[N_t] &\ge \Big( \prod_{i \in B} L_i \Big) E_0\Big[ \int_0^{(\zeta_{t/d}-z)_+}G_U(u)\,du \Big] \\
	&\quad - C'\big(E_0[G_U(\zeta_{t/d} - z)1(\zeta_{t/d} > z) ] + E_0[ \widehat G_U(\zeta_{t/d} - z)1(\zeta_{t/d} > z) ] \big). 
\end{align*}
When combined with \eqref{eq:ENupperbound} this completes the proof of \eqref{eq:ENexpressionbound}.

Now suppose $\sup_{t\ge0} E_0[\int_0^{(\zeta_{t/d}-z)_+}G_U(u)\,du] < \infty$. For $H(u) = \int_0^u G_U(v)\,dv$, this and $t^{-1/2}z \to \infty$ are the hypotheses of Lemma \ref{smallerfunction}. We claim that indeed $t^{-1/2}z \to \infty$. If not, then along some subsequence denoted again by $\{t\}$, $(t/d)^{-1/2}z \to c < \infty$. Note that, since $G_U \ge 1$, for any $M > 0$ we have
\begin{align*}
	E_0\Big[\int_0^{(\zeta_{t/d}-z)_+}G_U(u)\,du\Big] &\ge E_0[(\zeta_{t/d} - z)_+] \\
	&\ge E_0[(\zeta_{t/d} - z)1(\zeta_{t/d} > z + M)] \ge MP_0(\zeta_{t/d} > z + M). 
\end{align*}
Then by the central limit theorem, 
\begin{align*}
	\liminf_{t\to\infty} E_0\Big[\int_0^{(\zeta_{t/d}-z)_+}G_U(u)\,du\Big] &\ge \liminf_{t\to\infty} MP_0\big((t/d)^{-1/2}\zeta_{t/d} > (t/d)^{-1/2}(z + M)\big) \\
	&= MP(X > c), 
\end{align*}
where we recall that $X$ is standard Gaussian. Since $P(X > c) > 0$, letting $M \to \infty$ gives a contradiction, and so we conclude that $t^{-1/2}z \to \infty$.

Now suppose that $\{g_i\}_{i \in U}$ satisfy Condition \ref{smallderiv}. 
By Remark \ref{G-rmk}, $G_U$ also satisfies Condition \ref{smallderiv}. Thus if $\tilde H = G_U$, then 
\begin{equation}\label{eq:G'/G}
	\lim_{u \to \infty} \frac{\tilde H'(u)}{H'(u)} = \lim_{u \to \infty} \frac{G_U'(u)}{G_U(u)} = 0. 
\end{equation}
Alternatively, if $\tilde H = \widehat G_U$, then 
\[
	\frac{\tilde H'(u)}{H'(u)} = \frac{\widehat G_U'(u)}{G_U(u)} = \frac{1}{G_U(u)} \sum_{i \in U} \prod_{l \in U \setminus \{i\}} (2g_l(u)+1) = \sum_{i \in U} \frac{1}{2g_i(u)+1} \to 0, 
\]
since $\{g_i\}_{i \in U}$ are unbounded. 
From Lemma \ref{smallerfunction}, we conclude that 
\[
	\lim_{t\to\infty} E_0[G_U(\zeta_{t/d} - z)1(\zeta_{t/d} > z)] = \lim_{t\to\infty} E_0[\widehat G_U(\zeta_{t/d} - z)1(\zeta_{t/d} > z)] = 0. \qedhere
\]
\end{proof}

\section{Proof of Theorem \ref{meta}}\label{mainthmproof}

Here we prove our main result, using Proposition \ref{ssbound} and Corollary \ref{Ccor}. Recall the assumptions that $g_2, \ldots, g_d$ satisfy Conditions \ref{nondecrcond}--\ref{smallderiv}, that 
\begin{equation}\label{eq:d=2,3ass}
	\gamma_d(t) ( E_0[ G(\zeta_{t/d} - z)1(\zeta_{t/d} > z)] )^2 = 0, \qquad d \in \{2, 3\}, 
\end{equation}
and that $\sup_t \E_{\{g_i\}}[N_t] < \infty$, which will be taken as given throughout this section. From Remark \ref{z/sqrtt}, this last assumption implies that $t^{-1/2}z \to \infty$. 

\begin{proof}[Proof of Theorem \ref{meta}]
From Proposition \ref{ssbound} and \eqref{eq:Gto0}, 
$\lim_{t\to\infty} \sS_t(\{g_i\},z) = 0$.

When $d \ge 4$, $\gamma_d(t) = 1$. So, the assumption \eqref{eq:d=2,3ass} along with \eqref{eq:Gto0} implies
\begin{equation}\label{eq:gammaG^2to0}
	\lim_{t\to\infty} \gamma_d(t) ( E_0[G(\zeta_{t/d} - z)1(\zeta_{t/d} \ge z)] )^2 = 0, 
\end{equation}
for each $d \ge 2$. Thus, from Corollary \ref{Ccor}, to show that $\sC_t(\{g_i\},z) \to 0$ it remains to establish that 
$\lim_{t\to\infty} \gamma_{d+1}(t) E_0[G(\zeta_{t/d} - z)1(\zeta_{t/d} \ge z)] = 0$,
for each $d \ge 2$. 

When $d \ge 3$, this is immediate from \eqref{eq:Gto0}. When $d = 2$, 
\[
	\gamma_{3}(t) E_0[G(\zeta_{t/2} - z)1(\zeta_{t/2} \ge z)] = \frac{\log t}{t^{1/4}} \big( \gamma_2(t) ( E_0[G(\zeta_{t/2} - z)1(\zeta_{t/2} \ge z)] )^2 \big)^{1/2} \to 0, 
\]
using \eqref{eq:gammaG^2to0}. Thus, $\lim_{t\to\infty} \sC_t(\{g_i\},z) = 0$. 

Finally, 
$\lim_{t\to\infty} \sE_t(\{g_i\},z) = \lim_{t\to\infty} \left( \sS_t(\{g_i\},z) + \sC_t(\{g_i\},z) \right) = 0$. 
\end{proof}

\section{Proofs of Propositions \ref{ssbound} and \ref{mainbound} and Corollary \ref{Ccor}}\label{mainboundpf} 

First we give the short proof of Proposition \ref{ssbound}, which is the estimate on the sum of squares term, $\sS_t(\{g_i\}, z)$. Recall the assumption that $g_2, \ldots, g_d$ satisfy Condition \ref{nondecrcond}.

\begin{proof}[Proof of Proposition \ref{ssbound}]
Under $\P_{\{g_i\}}$, 
\[
	\eta_t(k) = \sum_{\eta_0(y) = 1} 1(\xi_y(t) = k) =  \sum_{j \le 0} \underset{y_1 = j}
	{\sum_{y \in \sR_{\{g_i\}}}} 1(\xi_y(t) = k) \leq \sum_{j \le 0} \underset{y_1 = j}
	{\sum_{y \in \sR_{\{g_i\}}}} 1(\xi_y(t) \cdot e_1 = k_1). 
\]
Since each $g_i$ is nondecreasing, $G$ is as well. Then for $k_1 > z$, 
\begin{align*}
	\E_{\{g_i\}}[\eta_t(k)] &\leq \sum_{j \ge 0} \Big( \prod_{i=2}^d (2 \lfloor g_i(j) \rfloor + 1) \Big) P_{-j}(\zeta_{t/d} = k_1) \\
	&\le \sum_{j \ge 0} G(j) P_{-j}(\zeta_{t/d} = k_1) 
	= \sum_{j \ge 0} G(j) P_0(\zeta_{t/d} = k_1 + j)  \\
	&= E_0[ G(\zeta_{t/d} - k_1) 1(\zeta_{t/d} \ge k_1)]  \le E_0[ G(\zeta_{t/d} - z)1(\zeta_{t/d} > z)].
\end{align*}
The result then follows from 
\[
	\sS_t(\{g_i\}, z) = \underset{k_1 > z}{\sum_{k \in \Z^d}}  (\E_{\{g_i\}}[\eta_t(k)])^2 \le \E_{\{g_i\}}[N_t] \cdot \underset{k_1 > z}{\sup_{k \in \Z^d}} \E_{\{g_i\}}[\eta_t(k)]. \qedhere
\]
\end{proof}

Next, we prove Proposition \ref{mainbound}, the bound on the sum of covariances, $\sC_t(\{g_i\},z)$. From Lemma \ref{covintrep}, 
\begin{align}\nonumber
	&\sC_t(\{g_i\}, z) \\ \nonumber
	&\le \frac{1}{d} \sum_{j \in \Z^d} \sum_{i=1}^d \int_0^t P_j(\zeta_{t-s}^{(d)} \cdot e_1 \ge z)^2\left(E_j[\eta_{\{g_l\}}(\zeta^{(d)}_s)] - E_{j + e_i}[\eta_{\{g_l\}}(\zeta^{(d)}_s)]\right)^2\,ds \\ \label{eq:covintegral}
	&= \frac{1}{d} \sum_{j \in \Z} \int_0^t P_j(\zeta_{(t-s)/d} \ge z)^2 \sum_{i=1}^d \sum_{k \in \Z^{d-1}} \left(E_{(j,k)}[\eta_{\{g_l\}}(\zeta^{(d)}_s)] - E_{(j,k) + e_i}[\eta_{\{g_l\}}(\zeta^{(d)}_s)]\right)^2\,ds, 
\end{align}
where we recall that $\zeta^{(d)}_t$ is a continuous time simple random walk on $\Z^d$ and $\zeta_t = \zeta_t^{(1)}$ is a continuous time simple random walk on $\Z$. Recall also that $P_j$ is the measure under which the random walk is at $j$ at time $0$, with corresponding expectation denoted $E_j$.

The main estimate for the proof of Proposition \ref{mainbound} is given in the following lemma. Recall the functions $G$ and $\widehat G$ from Definition \ref{capitalfunctions}. The estimate is stated using the following notation: For $j \in \Z$, $t \ge 0$, and $f : \R_+ \to \R_+$, let 
\[
	\mu_{j,t}(f) = E_j[f(-\zeta_{t/d})1(\zeta_{t/d} \le 0)], 
\]
when the expectation exists. 

\begin{lemma}\label{positivealphabound} 
Suppose $g_2, \ldots, g_d$ satisfy Conditions \ref{nondecrcond} and \ref{covcond}. 
For some constant $C > 0$ and any $j \in \Z$, 
\begin{align*}
	&\sum_{i=1}^d 	\sum_{k \in \Z^{d-1}} \left(E_{(j,k)}[\eta_{\{g_l\}}(\zeta^{(d)}_s)] - E_{(j,k)+e_i}[\eta_{\{g_l\}}(\zeta^{(d)}_s)]\right)^2 \\
	&\le  C \Big(\frac{G(0)^2P_j(\zeta_{s/d} = 0)^2 + \mu_{j,s}(G')^2}{ (1 \vee s)^{(d-1)/2} } + \frac{ \mu_{j,s}(G)\mu_{j,s}(\widehat G')}{(1 \vee s)^{d/2}} \Big). 
\end{align*}. 
\end{lemma}

Before providing the proof of the above lemma, we establish several estimates that are used therein. 
Throughout this section, we will denote elements $k \in \Z^{d-1}$ with shifted indices, namely $k = (k_2, \ldots, k_d)$. This way the indices match those for $(j, k) \in \Z^d$ when $j \in \Z$. 

\begin{lemma}\label{Hboundlemma} Suppose $g_2, \ldots, g_d$ satisfy Conditions \ref{nondecrcond} and \ref{covcond}. 
For $2 \le i \le d$, $s > 0$, $k \in \Z^{d-1}$, and integer $m \ge 1$, define 
\[
	H^i_{k,s}(m) = P_{k_i}(g_i(m-1) < |\zeta_{s/d}| \le g_i(m)) \prod_{l \in \{2, \ldots, d\} \setminus \{i\}} P_{k_l}(|\zeta_{s/d}| \le g_l(m)), 
\]
and for $m \ge 0$, 
\begin{align*}
	\widehat H^i_{k,s}(m) &= \left| P_{k_i}(g_i(m) - 1 < \zeta_{s/d} \le g_i(m)) - P_{k_i}(-g_i(m) - 1 \le \zeta_{s/d} < - g_i(m)) \right| \\
	&\qquad \times \prod_{l \in \{2, \ldots, d\} \setminus \{i\}} P_{k_l}(|\zeta_{u/d}| \le g_l(m)). 
\end{align*}
There is $C>0$ so that the following bounds hold for each $m$ and $s$. First, 
\begin{equation}\label{eq:Hbounds}
	\sum_{i=2}^d \sup_{k \in \Z^{d-1}} H^i_{k,s}(m) \le C(1 \vee s)^{-(d-1)/2} G'(m)  \qquad \mbox{and} \qquad
	\sum_{i=2}^d \sum_{k \in \Z^{d-1}} H^i_{k,s}(m) \le C G'(m). 
\end{equation}
Second, 
\begin{equation}\label{eq:tildeHbounds}
\begin{aligned}
	&\sum_{i =2}^d \sup_{k \in \Z^{d-1}} \widehat H^i_{k,s}(m) \le C(1 \vee s)^{-(d-1)/2} \widehat G'(m), 
	\qquad \mbox{and} \\
&\sum_{i =2}^d \sum_{k \in \Z^{d-1}} \widehat H^i_{k,s}(m) \le C (1 \vee s)^{-1/2} G(m).
\end{aligned}
\end{equation}
\end{lemma}

\begin{proof} We use two general random walk estimates: There is $C > 0$ such that 
\begin{equation}\label{eq:supsquareroot}
	\sup_{a, b \in \Z} P_a(\zeta_t = b) \le \frac{C}{\sqrt{t}}, 
\end{equation}
and for any $b \in \Z$, 
\begin{equation}\label{eq:diffest}
	\sum_{a \in \Z} \left| P_0(\zeta_t = a) - P_0(\zeta_t = a + b) \right| \le  C \min \Big\{1,  \frac{|b|}{\sqrt{t}} \Big\}. 
\end{equation}
Proofs can be found in \cite{LawLim10} and the Appendix of \cite{ConSet23}. 
These bounds applied to the following quantities will imply the result of the lemma:
For each $n \in \Z$, let 
\begin{align*}
	h^i_{n,s}(m) &= P_{n}(g_i(m-1) < |\zeta_{s/d}| \le g_i(m)), \\ 
	\hat h^i_{n,s}(m)&= \left| P_n(g_i(m) - 1 < \zeta_{s/d} \le g_i(m)) - P_n(-g_i(m) - 1 \le \zeta_{s/d} < -g_i(m)) \right|, \qquad\text{and} \\
	f^i_{n,s}(m) &= P_{n}(|\zeta_{s/d}| \le g_i(m)). 
\end{align*}

The bound \eqref{eq:supsquareroot}
implies 
\begin{equation}\label{eq:supf}
\begin{aligned}
	\sup_n f^i_{n,s}(m) &= \sup_n \sum_{-g_i(m) \le l \le g_i(m)} P_n(\zeta_{s/d} = l) \\
	&\le (2g_i(m)+1)\min\Big\{ 1, \frac{C}{\sqrt{s/d}} \Big\} \le C'\frac{2g_i(m)+1}{(1 \vee s)^{1/2}}, 
\end{aligned}
\end{equation}
for each $i$.
On the other hand, 
\begin{equation}\label{eq:sumf}
	\sum_{n \in \Z} f^i_{n,s}(m) = \sum_{-g_i(m) \le l \le g_i(m)} \sum_{n \in \Z} P_0(\zeta_{s/d} = l - n) \le 2g_i(m) + 1. 
\end{equation}

The bound \eqref{eq:supsquareroot} and Condition \ref{covcond} give 
\begin{equation}\label{eq:suph}
\begin{aligned}
	\sup_n h^i_{n,s}(m) &\le  \sum_{g_i(m-1) < l \le g_i(m)} \sup_n P_n(\zeta_{s/d} = l) + \sum_{-g_i(m) \le l < -g_i(m-1)} \sup_n P_n(\zeta_{s/d} = l) \\
	&\le \frac{C'}{(1 \vee s)^{1/2}} (g_i(m) - g_i(m-1)) 
	 \le \frac{C''}{(1 \vee s)^{1/2}} g_i'(m). 
\end{aligned}
\end{equation}
(Note that in the case $g_i(m-1) = g_i(m)$, $\sup_n h^i_{n,s}(m) = 0$ and the above bound is still valid.)   
By similar arguments, 
\begin{equation}\label{eq:sumh}
\begin{aligned}
	\sum_{n\in\Z} h^i_{n,s}(m) &= \sum_{g_i(m-1) < l \le g_i(m)} \sum_{n \in \Z} P_n(\zeta_{s/d} = l) \le g_i(m) - g_i(m-1) \le C g_i'(m). 
\end{aligned}
\end{equation}

Next, note that 
\begin{align*}
	\hat h^i_{n,s}(m) &=  \left| P_n(\zeta_{s/d} = \lfloor g_i(m) \rfloor) - P_n(\zeta_{s/d} = \lceil - g_i(m) - 1 \rceil) \right| \\
	&=  \left| P_0(\zeta_{s/d} = \lfloor g_i(m) \rfloor - n) - P_0(\zeta_{s/d} = \lceil - g_i(m) - 1 \rceil - n) \right|. 
\end{align*}
Then \eqref{eq:supsquareroot} implies 
\begin{equation}\label{eq:suptildeh}
	\sup_n \hat h^i_{n,s}(m) \le \frac{C'}{(1 \vee s)^{1/2}}, 
\end{equation}
and \eqref{eq:diffest} implies
\begin{equation}\label{eq:sumtildeh}
\begin{aligned}
	\sum_n \hat h^i_{n,s}(m) &\le C \min\Big\{1,  \frac{(\lfloor g_i(m) \rfloor - \lceil - g_i(m) - 1 \rceil )}{\sqrt{s}} \Big\}  \le C' \frac{2g_i(m)+1}{(1 \vee s)^{1/2}}. 
\end{aligned}
\end{equation}

Now, it follows from \eqref{eq:supf} and \eqref{eq:suph} that 
\begin{align*}
	\sup_{k \in \Z^{d-1}} H^i_{k,s}(m) &= \sup_{k \in \Z^{d-1}} h^i_{k_i,s}(m) \prod_{l \ne i} f^l_{k_l,s}(m) \\
	&\le \Big(\sup_{n \in \Z} h^i_{n,s}(m)\Big) \prod_{l \ne i} \sup_{n \in \Z} f^l_{n,s}(m) \le \frac{C}{(1 \vee s)^{(d-1)/2}} g_i'(m) \prod_{l \ne i} (2g_l(m)+1). 
\end{align*}
Thus, 
\begin{align*}
	\sum_{i=2}^d \sup_{k \in \Z^{d-1}} H^i_{k,s}(m) &\le \frac{C}{(1 \vee s)^{(d-1)/2}} \sum_{i=2}^d g_i'(m) \prod_{l \ne i} (2g_l(m) +1)\\
	&= \frac{C/2}{(1 \vee s)^{(d-1)/2}} \Big(\prod_{i=2}^d (2g_i(m)+1)\Big)' =  \frac{C/2}{(1 \vee s)^{(d-1)/2}}G'(m). 
\end{align*}
This establishes the first inequality in \eqref{eq:Hbounds}. Furthermore, from \eqref{eq:sumf} and \eqref{eq:sumh} we obtain 
\begin{align*}
	\sum_{k \in \Z^{d-1}} H^i_{k,s}(m) &= \sum_{k \in \Z^{d-1}} h^i_{k_i,s}(m) \prod_{l \ne i} f^l_{k_l,s}(m) \\
	&= \Big( \sum_{n \in \Z} h^i_{n,s}(m) \Big) \prod_{l \ne i} \sum_{n \in \Z} f^l_{n,s}(m) \le C g_i'(m) \prod_{l \ne i} (2g_l(m)+1), 
\end{align*}
so that 
$\sum_{i=2}^d \sum_{k \in \Z^{d-1}} H^i_{k,s}(m) \le C G'(m)$, which is
the second inequality in \eqref{eq:Hbounds}. 

The bounds involving $\widehat H^i_{k,s}(m)$ in \eqref{eq:tildeHbounds} follow similarly. First, using \eqref{eq:supf} and \eqref{eq:suptildeh},   
\begin{align*}
	\sum_{i = 2}^d \sup_{k \in \Z^{d-1}} \widehat H^i_{k,s}(m) &= \sum_{i = 2}^d \sup_{k \in \Z^{d-1}} \hat h^i_{k_i,s}(m) \prod_{l \in\{2, \ldots, d\}\setminus\{i\}} f^l_{k_l,s}(m) \\
	&\le  \frac{C}{(1 \vee s)^{(d-1)/2}} \sum_{i = 2}^d \prod_{l \ne i} (2g_l(m)+1) = \frac{C\widehat G'(m)}{(1 \vee s)^{(d-1)/2}}. 
\end{align*}
Moreover, using \eqref{eq:sumf} and \eqref{eq:sumtildeh}, 
\begin{align*}
	\sum_{i = 2}^d \sum_{k \in \Z^{d-1}} \widehat H^i_{k,s}(m) &= \sum_{i=2}^d \Big(\sum_{n\in \Z} \hat h^i_{n,s}(m) \Big) \prod_{l \in \{2, \ldots, d\}\setminus\{i\}} \sum_{n \in \Z}f^l_{n,s}(m) \\
	&\le C \sum_{i=2}^d \frac{2g_i(m)+1}{(1 \vee s)^{1/2}}  \prod_{l \ne i} (2g_l(m)+1) = C' \frac{G(m)}{(1 \vee s)^{1/2}}. 
	\qedhere
\end{align*}
\end{proof}

Now we turn to the proof of Lemma \ref{positivealphabound}. 

\begin{proof}[Proof of Lemma \ref{positivealphabound}] This proof has two parts. First, we show that, for each $j \in \Z$ and $s > 0$, 
\begin{equation}\label{eq:i=1}
\begin{aligned}
	&\sum_{k \in \Z^{d-1}} \left(E_{(j,k)}[\eta_{\{g_l\}}(\zeta^{(d)}_s)] - E_{(j+1,k)}[\eta_{\{g_l\}}(\zeta^{(d)}_s)]\right)^2 \\
	&\le C(1 \vee s)^{-(d-1)/2} \big(  G(0)^2P_j(\zeta_{s/d} = 0)^2 + (E_j[G'(-\zeta_{s/d})1(\zeta_{s/d} < 0)])^2  \big). 
\end{aligned}
\end{equation}
Second, we show 
\begin{equation}\label{eq:i=2:d}
\begin{aligned}
	&\sum_{i=2}^d 	\sum_{k \in \Z^{d-1}} \left(E_{(j,k)}[\eta_{\{g_l\}}(\zeta^{(d)}_s)] - E_{(j,k)+e_i}[\eta_{\{g_l\}}(\zeta^{(d)}_s)]\right)^2 \\
	&\le C (1 \vee s)^{-d/2} E_j[G(-\zeta_{s/d})1(\zeta_{s/d} \le 0)] E_j[\widehat G'(-\zeta_{s/d})1(\zeta_{s/d} \le 0)] . 
\end{aligned}
\end{equation}
Together \eqref{eq:i=1} and \eqref{eq:i=2:d} imply the result. 

\medskip

\textbf{Proof of \eqref{eq:i=1}.} Recall that $\sH_d$ denotes the half-space $\sH_d = \{x \in \Z^d : x_1 \le 0\}$. 
For each $i = 2, \ldots, d$, define the set
\begin{equation}\label{eq:Ri}
	\sR^{(i)} = \{x \in \sH_d : -g_l(-x_1) \le x_l \le g_l(-x_1), \; l \in \{2, \ldots, d\} \setminus \{i\}\}. 
\end{equation}
(When $d=2$, $\sR^{(2)}= \sH_2$.)   
Since the $\{g_l\}$ are nondecreasing, 
$\sR_{\{g_l\}} - e_1 \subset \sR_{\{g_l\}}$. Moreover, 
\begin{align*}
	\sR_{\{g_l\}} \setminus (\sR_{\{g_l\}} - e_1) &= \{x_1 = 0, |x_i| \le g_i(0), i = 2, \ldots, d\} \\
	&\quad \cup \Big[  \{x_1 < 0\} \cap \bigcup_{i=2}^d \left( \sR^{(i)} \cap \{g_i(-x_1 -1) < |x_i| \le g_i(-x_1)\} \right) \Big]. 
\end{align*}
This identity is depicted in Figure \ref{fig:covcalc} (a) for the $d = 2$ case. 

\begin{figure}[t]
\captionsetup{width=.9\linewidth}
\centering
\qquad\quad
\hspace*{\fill}%
\begin{subfigure}[b]{0.3\textwidth}
\centering
\begin{tikzpicture}[scale=0.6]

\draw[->] (-6.5,0) -- (1,0) node[below] {\footnotesize $x_1$};
\draw[->] (0,-3.5) -- (0,3.5) node[right] {\footnotesize $x_2$};

\pgfmathsetmacro{\p}{2}
\pgfmathsetmacro{\m}{0.05}
\pgfmathsetmacro{\b}{4}
\pgfmathsetmacro{\s}{0.25}

\draw[scale=\s, domain=-14.5:0, smooth, variable=\x, very thick] plot ({\x},{\m*(-\x)^(\p) + \b});
\draw[scale=\s, domain=-14.5:0, smooth, variable=\x, very thick] plot ({\x},{-\m*(-\x)^(\p) - \b});
\draw[smooth, very thick] (0,\s*\b) -- (0,-\s*\b);

\draw[scale=\s, domain=-17.5:-3, smooth, variable=\x, very thick, dashed] plot ({\x},{\m*(-\x-3)^(\p) + \b});
\draw[scale=\s, domain=-17.5:-3, smooth, variable=\x, very thick, dashed] plot ({\x},{-\m*(-\x-3)^(\p) - \b});
\draw[smooth, very thick, dashed] (-\s*3,\s*\b) -- (-\s*3,-\s*\b);

\node[right] at (-\s*13.5, 3.5) {\footnotesize $g_2(-x_1)$};
\node[right] at (-\s*13.5, -3.5) {\footnotesize $-g_2(-x_1)$};

\node[left] at (-\s*17.5, 3.5) {\footnotesize $g_2(-x_1-1)$};
\node[left] at (-\s*17.5, -3.5) {\footnotesize $-g_2(-x_1-1)$};

\draw[->,ultra thick] (0,0) -- (-\s*3,0) node[above left] {\footnotesize $-e_1$};

\node[left] at (-\s*12,-\s*\b) {\footnotesize $\sR_{g_2} \setminus (\sR_{g_2} - e_1)$};
\draw (-\s*12,-\s*\b) -- (-\s*9.5, -\s*\b*1.85) ;

\end{tikzpicture}
\caption{}
\end{subfigure}\qquad\qquad\hfill%
\begin{subfigure}[b]{0.3\textwidth}
\centering
\begin{tikzpicture}[scale=0.6]

\draw[->] (-6.5,0) -- (1,0) node[below] {\footnotesize $x_1$};
\draw[->] (0,-3.5) -- (0,3.5) node[right]  {\footnotesize $x_2$};

\pgfmathsetmacro{\p}{0.55}
\pgfmathsetmacro{\m}{3}
\pgfmathsetmacro{\b}{0.5}
\pgfmathsetmacro{\s}{0.25}

\draw[scale=\s, domain=-16:0, smooth, variable=\x, very thick] plot ({\x},{\m*(-\x)^(\p) + \b});
\draw[scale=\s, domain=-16:0, smooth, variable=\x, very thick] plot ({\x},{-\m*(-\x)^(\p) - \b});

\draw[scale=\s, domain=-16:0, smooth, variable=\x, very thick, dashed] plot ({\x},{\m*(-\x)^(\p) + \b - 4});
\draw[scale=\s, domain=-6:0, smooth, variable=\x, very thick, dashed] plot ({\x},{-\m*(-\x)^(\p) - \b - 4});

\node[left] at (-\s*16, 3.5) {\footnotesize $g_2(-x_1)$};
\node[left] at (-\s*16, 3.5 - \s*4) {\footnotesize $g_2(-x_1)-1$};

\node[left] at (-\s*16, -3.5) {\footnotesize $-g_2(-x_1)$};
\node at (-\s*8, -3.5) {\footnotesize $-g_2(-x_1)-1$};

\draw (-\s*5, \s*5) -- (-\s*9, \s*3) node[left] {\footnotesize $A$};
\draw (-\s*5, -\s*9) -- (-\s*9, -\s*7) node[left] {\footnotesize $B$};

\draw[->,ultra thick] (0,0) -- (0,-\s*4) node[right] {\footnotesize $-e_2$};

\end{tikzpicture}
\caption{}
\end{subfigure}%
\qquad\qquad
\hspace*{\fill}%
\caption{\footnotesize Illustrations of the calculations in the proof of Lemma \ref{positivealphabound} in $\Z^2$ with initial profile $\eta_{g_2}(x) = 1(x \in \sR_{g_2})$, $\sR_{g_2} = \{x : x_1 \le 0, |x_2| \le g_2(-x_1)\}$. The boundary of $\sR_{g_2}$ is shown as a solid line. 
(a) The region $\sR_{g_2} \setminus (\sR_{g_2} - e_1)$, where the boundary of $\sR_{g_2} - e_1$ is shown as a dashed line. (b) $1_{\sR_{g_2}} - 1_{\sR_{g_2} - e_2} = 1_A - 1_B$, where $A = \{x : x_1 \le 0, g_2(-x_1)-1 < x_2 \le g_2(-x_1)\}$ and $B = \{x : x_1 \le 0, -g_2(-x_1)-1 \le x_2 < -g_2(-x_1)\}$. The boundary of $\sR_{g_2}-e_2$ is shown as a dashed line.}\label{fig:covcalc}
\end{figure}
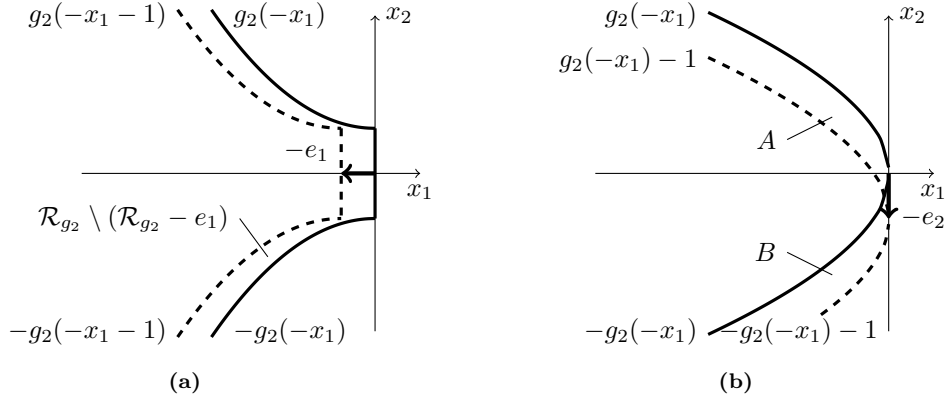

Using \eqref{eq:rwindependentdecomp}, for $j \in \Z$ and $k \in \Z^{d-1}$ 
we have 
\begin{align*}
	&E_{(j,k)}[\eta_{\{g_l\}}(\zeta^{(d)}_s)] - E_{(j+1,k)}[\eta_{\{g_l\}}(\zeta^{(d)}_s)] 
	= P_{(j,k)}(\zeta^{(d)}_s \in \sR_{\{g_l\}}) - P_{(j+1,k)}(\zeta^{(d)}_s \in \sR_{\{g_l\}}) \\
	&= P_{(j,k)}(\zeta^{(d)}_s \in \sR_{\{g_l\}} \setminus (\sR_{\{g_l\}} - e_1)) \\
	&\le P_{(j,k)}(\zeta_s^{(d)} \cdot e_1 = 0,  |\zeta_s^{(d)} \cdot e_i| \le g_i(0), i \in \{2, \ldots, d\}) \\
	&\quad + \sum_{i=2}^d P_{(j,k)} (\zeta_s^{(d)} \cdot e_1 < 0, g_i( -\zeta_s^{(d)} \cdot e_1-1) < |\zeta_s^{(d)} \cdot e_i| \le g_i(-\zeta_s^{(d)}\cdot e_1), \zeta^{(d)}_s \in \sR^{(i)})  \\
	&= P_j(\zeta_{s/d} = 0) \prod_{i=2}^d P_{k_{i}}( |\zeta_{s/d}| \le g_l(0)) \\
	&\quad + \sum_{i=2}^d \sum_{m > 0} P_j(\zeta_{s/d} = -m)  P_{k_{i}}(g_i(m-1) < |\zeta_{s/d}| \le g_i(m)) \prod_{l \in \{2, \ldots, d\}\setminus \{i\}} P_{k_{l}}(|\zeta_{s/d}| \le g_l(m)). 
\end{align*}
(When $d=2$, the empty product above is by convention set to $1$.)

In the notation of Lemma \ref{Hboundlemma}, this says that
\begin{align*}
	0 &\le E_{(j,k)}[\eta_{\{g_l\}}(\zeta_s^{(d)})] - E_{(j+1,k)}[\eta_{\{g_l\}}(\zeta_s^{(d)})] \\
	&\le P_j(\zeta_{s/d} = 0) \prod_{i=2}^d P_{k_{i}}(|\zeta_{s/d}| \le g_i(0)) + \sum_{i=2}^dE_j[ H^i_{k,s}(-\zeta_{s/d})1(\zeta_{s/d} < 0)], 
\end{align*}
which implies 
\begin{equation}
\label{eq:sumHsquared}
\begin{aligned}
	& \sum_{k \in \Z^{d-1}} \big(E_{(j,k)}[\eta_{\{g_l\}}(\zeta_s^{(d)})] - E_{(j+1,k)}[\eta_{\{g_l\}}(\zeta_s^{(d)})]\big)^2 \\
	& \le 2  \sum_{k \in \Z^{d-1}}  P_j(\zeta_{s/d} = 0)^2 \prod_{i=2}^d P_{k_{i}}(|\zeta_{s/d}| \le g_i(0))^2
	 + 2 \sum_{k \in \Z^{d-1}}  \Big( \sum_{i=2}^dE_j[ H^i_{k,s}(-\zeta_{s/d})1(\zeta_{s/d} < 0)]\Big)^2. 
\end{aligned}
\end{equation}
We now estimate the first term on the right hand side of the previous display.
Using the inequality in \eqref{eq:supf}, 
\[
	P_{k_i}(|\zeta_{s/d}| \le g_i(0)) \le \frac{C(2g_i(0)+1)}{(1 \vee s)^{1/2}}, 
\]
for each $i$. Then, using the above bound followed by the bound in \eqref{eq:sumf}, 
\begin{equation}\label{eq:Hsquaredbound1}
\begin{aligned}
	& \sum_{k \in \Z^{d-1}}  P_j(\zeta_{s/d} = 0)^2 \prod_{i=2}^d P_{k_{i}}(|\zeta_{s/d}| \le g_i(0))^2 \\
	&\le \frac{CG(0)P_j(\zeta_{s/d} = 0)^2}{(1 \vee s)^{(d-1)/2}}  \prod_{i=2}^d \sum_{k_i \in \Z} P_{k_i}(|\zeta_{s/d}| \le g_i(0)) \le \frac{C'G(0)^2P_j(\zeta_{s/d} = 0)^2}{(1 \vee s)^{(d-1)/2}}. 
\end{aligned}
\end{equation}

Next, we bound the second quantity in \eqref{eq:sumHsquared} using \eqref{eq:Hbounds}: 
\begin{align} \nonumber
	&\sum_{k \in \Z^{d-1}} \Big( \sum_{i=2}^d E_j[ H^i_{k,s}(-\zeta_{s/d})1(\zeta_{s/d} < 0)] \Big)^2 \\ \nonumber 
	&\le  E_j\Big[  \sum_{i=2}^d \sup_{k \in \Z^{d-1}} H^i_{k,s}(-\zeta_{s/d})1(\zeta_{s/d} < 0)\Big] E_j\Big[  \sum_{i=2}^d\sum_{k \in \Z^{d-1} } H^i_{k,s}(-\zeta_{u/d})1(\zeta_{s/d} < 0)\Big] \\ \label{eq:Hsquaredbound2}
	&\le \frac{C}{(1 \vee s)^{(d-1)/2}} \left( E_j[ G'(-\zeta_{s/d}) 1(\zeta_{s/d} < 0) ] \right)^2. 
\end{align}
Together \eqref{eq:sumHsquared}, \eqref{eq:Hsquaredbound1}, and \eqref{eq:Hsquaredbound2} establish \eqref{eq:i=1}. 

\medskip

 \textbf{Proof of \eqref{eq:i=2:d}.} Let $i \in \{2, \ldots, d\}$. Recalling the notation $\sR^{(i)}$ from \eqref{eq:Ri}, 
\begin{align*}
	&1(x \in \sR_{\{g_l\}}) - 1(x \in \sR_{\{g_l\}} - e_i)  \\
	&= 1(x_1 \le 0, x \in \sR^{(i)}) \big( 1(g_i(-x_1) - 1 < x_i \le g_i(-x_1)) 
	- 1(-g_i(-x_1) - 1 \le x_i < -g_i(-x_1)) \big). 
\end{align*}
This identity is depicted in Figure \ref{fig:covcalc} (b) for the $d=2$ case. 

Then, using \eqref{eq:rwindependentdecomp}, for $(j,k) \in \Z^d$, 
\begin{align*}
	&\big| E_{(j,k)} [\eta_{\{g_l\}}(\zeta_s^{(d)})] - E_{(j,k) + e_i}[\eta_{\{g_l\}}(\zeta_s^{(d)})] \big| = \big| P_{(j,k)}(\zeta_s^{(d)} \in \sR_{\{g_l\}}) - P_{(j,k)}(\zeta_s^{(d)} \in \sR_{\{g_l\}} - e_i) \big| \\
	&= \big| P_{(j,k)}(\zeta_s^{(d)} \cdot e_1 \le 0,  g_i(-\zeta_s^{(d)} \cdot e_1) - 1 < \zeta_s^{(d)} \cdot e_i \le  g_i(-\zeta_s^{(d)} \cdot e_1), \zeta_s^{(d)} \in \sR^{(i)})  \\
	&\qquad - P_{(j,k)}(\zeta_s^{(d)} \cdot e_1 \le 0, -g_i(-\zeta_s^{(d)} \cdot e_1) - 1 \le \zeta_s^{(d)} \cdot e_i < - g_i(-\zeta_s^{(d)} \cdot e_1), \zeta_s^{(d)} \in \sR^{(i)}) \big| \\
	&\le \sum_{m \ge 0} P_j(\zeta_{s/d} = -m) \big|P_{k_{i}}(g_i(m) - 1 < \zeta_{s/d} \le g_i(m))  \\
	&\qquad -  P_{k_{i}}(-g_i(m)-1 \le \zeta_{s/d} < -g_i(m)) \big| \prod_{l \in \{2, \ldots, d\}\setminus \{i\}} P_{k_{l}}(-g_l(m) \le \zeta_{s/d} \le g_l(m)) \\
	&= E_j[ \widehat H_{k,s}^i(-\zeta_{s/d})1(\zeta_{s/d} \le 0) ],
\end{align*}
where the last equality uses the notation from Lemma \ref{Hboundlemma}.

It follows that 
\begin{equation}\label{eq:secondbound}
\begin{aligned}
	&\sum_{i=1}^d 	\sum_{k \in \Z^{d-1}} \big(E_{(j,k)}[\eta_{\{g_l\}}(\zeta_s^{(d)})] - E_{(j,k)+e_i}[\eta_{\{g_l\}}(\zeta_s^{(d)})]\big)^2 \\
		& \le \sum_{i = 2}^d \sum_{k \in \Z^{d-1}} \big( E_j[ \widehat H^i_{k,s}(-\zeta_{s/d})1(\zeta_{s/d} \le 0)] \big)^2. 
\end{aligned}
\end{equation}

Now, applying both inequalities in \eqref{eq:tildeHbounds} to the quantity on the right hand side of the above display, 
\begin{align*}
	&\sum_{i = 2}^d \sum_{k \in \Z^{d-1}} \big( E_j[ \widehat H^i_{k,s}(-\zeta_{s/d})1(\zeta_{s/d} \le 0)] \big)^2 \\
	&\le \sum_{i=2}^d E_j\Big[ \sup_{k \in \Z^{d-1}} \widehat H^i_{k,s}(-\zeta_{s/d})1(\zeta_{s/d} \le 0) \Big] E_j\Big[ \sum_{k \in \Z^{d-1}}\widehat H^i_{k,s}(-\zeta_{s/d})1(\zeta_{s/d} \le 0) \Big] \\
	&\le  E_j\Big[ \sum_{i=2}^d\sup_{k \in \Z^{d-1}} \widehat H^i_{k,s}(-\zeta_{s/d})1(\zeta_{s/d} \le 0) \Big] E_j\Big[\sum_{i=2}^d \sum_{k \in \Z^{d-1}} \widehat H^i_{k,s}(-\zeta_{s/d})1(\zeta_{s/d} \le 0) \Big] \\
	&\le \frac{C}{(1 \vee s)^{d/2}} E_j[ \widehat G'(-\zeta_{s/d})1(\zeta_{s/d} \le 0)] E_j [G(-\zeta_{s/d}) 1(\zeta_{s/d} \le 0) ]. 
\end{align*}
Combined with \eqref{eq:secondbound}, this completes the proof of \eqref{eq:i=2:d}.
\end{proof}

Now we prove Proposition \ref{mainbound}. 

\begin{proof}[Proof of Proposition \ref{mainbound}]
Recall that $\gamma_2(t) = \sqrt{t}$, $\gamma_3(t) = \log t$, and $\gamma_d(t) = 1$ for $d \ge 4$. There is a constant $C>0$ so that for all $d$ and all sufficiently large $t$, 
\begin{equation}\label{eq:gammaintegral}
	\gamma_d(t) \ge C \int_0^t \frac{ds}{(1 \vee s)^{(d-1)/2}}. 
\end{equation}

 Now, beginning from \eqref{eq:covintegral}, 
applying Lemma \ref{positivealphabound}, then using $\sum_j a_j^2 \le (\sum_j a_j)^2$ for $a_j \ge 0$, we obtain 
\begin{align} \nonumber 
	\sC_t(\{g_i\}, z) 
	&\le C \sum_{j \in \Z} \int_0^t P_j(\zeta_{(t-s)/d} \ge z)^2 \Big(\frac{G(0)^2P_j(\zeta_{s/d} = 0)^2 + \mu_{j,s}(G')^2}{ (1 \vee s)^{(d-1)/2} }  + \frac{ \mu_{j,s}(G)\mu_{j,s}(\widehat G')}{(1 \vee s)^{d/2}} \Big)\,ds  \\ \nonumber
	&\le C \int_0^t \bigg[ \Big( \sum_{j \in \Z} G(0) P_j(\zeta_{s/d} = 0)P_j(\zeta_{(t-s)/d} \ge z) \Big)^2 \\ \label{eq:firstCbound}
	&\qquad\qquad\quad +  \Big( \sum_{j \in \Z} \mu_{j,s}(G')P_j(\zeta_{(t-s)/d} \ge z) \Big)^2 \bigg] \frac{ds}{(1 \vee s)^{(d-1)/2}} \\ \nonumber
	&\quad + C \int_0^t \Big( \sum_{j \in \Z} \mu_{j,s}(G)\mu_{j,s}(\widehat G')P_j(\zeta_{(t-s)/d} \ge z)^2\Big) \frac{ds}{(1 \vee s)^{d/2}}.  
\end{align}
 
From the first equality in Lemma \ref{jsum}, 
\begin{equation}\label{eq:jsumG(0)}
	 \sum_{j \in \Z} G(0) P_j(\zeta_{s/d} = 0)P_j(\zeta_{(t-s)/d} \ge z) = G(0) P_0(\zeta_{t/d} \ge z), 
\end{equation}
and 
\begin{equation}\label{eq:jsumG'}
\begin{aligned}
	\sum_{j \in \Z} \mu_{j,s}(G')P_j(\zeta_{(t-s)/d} \ge z) &= \sum_{j \in \Z} E_j[ G'(-\zeta_{s/d})1(\zeta_{s/d} \le 0)]P_j(\zeta_{(t-s)/d} \ge z) \\
	&= \sum_{k \le 0} G'(-k) P_k(\zeta_{t/d} \ge z). 
\end{aligned}
\end{equation}
Moreover, 
\begin{align*}
	 &\sum_{j \in \Z} \mu_{j,s}(G)\mu_{j,s}(\widehat G')P_j(\zeta_{(t-s)/d} \ge z)^2 \\
	 &= \sum_{j \in \Z} E_j[G(- \zeta_{s/d})1(\zeta_{s/d} \le 0)]E_j[ \widehat G'(-\zeta_{s/d})1(\zeta_{s/d} \le 0)]P_j(\zeta_{(t-s)/d} \ge z)^2 \\
	 &\le \sup_{j \in \Z} E_j[G(- \zeta_{s/d})1(\zeta_{s/d} \le 0)]P_j(\zeta_{(t-s)/d} \ge z) \cdot \sum_{j \in \Z} E_j[ \widehat G'(-\zeta_{s/d})1(\zeta_{s/d} \le 0)]P_j(\zeta_{(t-s)/d} \ge z). 
\end{align*}
Since $G$ is nondecreasing, applying both statements of Lemma \ref{jsum} gives
\[
	\sum_{j \in \Z} E_j[ \widehat G'(-\zeta_{s/d})1(\zeta_{s/d} \le 0)]P_j(\zeta_{(t-s)/d} \ge z) = \sum_{k \le 0} \widehat G'(-k) P_k(\zeta_{t/d} \ge z), 
\]
and 
\[
	\sup_{j \in \Z} E_j[G(- \zeta_{s/d})1(\zeta_{s/d} \le 0)]P_j(\zeta_{(t-s)/d} \ge z) \le  E_0[G(\zeta_{t/d} - z)1(\zeta_{t/d} \ge z)] . 
\]
It follows that 
\begin{equation}\label{eq:jsumGhatG'}
\begin{aligned}
	&\sum_{j \in \Z} \mu_{j,s}(G)\mu_{j,s}(\widehat G')P_j(\zeta_{(t-s)/d} \ge z)^2\\
	& \le  E_0[G(\zeta_{t/d} - z)1(\zeta_{t/d} \ge z)] \sum_{k \le 0} \widehat G'(-k) P_k(\zeta_{t/d} \ge z).
\end{aligned}
\end{equation}

Together, \eqref{eq:gammaintegral}, \eqref{eq:firstCbound}, \eqref{eq:jsumG(0)}, \eqref{eq:jsumG'}, and \eqref{eq:jsumGhatG'} imply, noting $G'\geq 0$, that for all sufficiently large $t$, 
\begin{align} \nonumber
	\sC_t(\{g_i\}, z) &\le C \bigg[ G(0)^2P_0(\zeta_{t/d} \ge z)^2 + \Big(\sum_{k \le 0} G'(-k) P_k(\zeta_{t/d} \ge z) \Big)^2 \bigg] \int_0^t \frac{ds}{(1 \vee s)^{(d-1)/2}} \\ \nonumber
	&\quad + C E_0[G(\zeta_{t/d} - z)1(\zeta_{t/d} \ge z)] \Big(\sum_{k \le 0} \widehat G'(-k) P_k(\zeta_{t/d} \ge z) \Big) \int_0^t \frac{ds}{(1 \vee s)^{d/2}} \\
	\begin{split}\label{eq:Cboundinsums}
		&\le C' \gamma_d(t) \Big(G(0)P_0(\zeta_{t/d} \ge z) + \sum_{k \ge 0} G'(k) P_{-k}(\zeta_{t/d} \ge z) \Big)^2 \\
		&\quad + C' \gamma_{d+1}(t) E_0[G(\zeta_{t/d} - z)1(\zeta_{t/d} \ge z)] \sum_{k \ge 0} \widehat G'(k) P_{-k}(\zeta_{t/d} \ge z). 
	\end{split}
\end{align}
It remains to bound the quantities in the above display in terms of the appropriate expectations. 

Applying Lemma \ref{sum2exp} and using that $G$ is nondecreasing gives 
\begin{equation}\label{eq:sumGbound}
\begin{aligned}
	&G(0)P_0(\zeta_{t/d} \ge z) + \sum_{k \ge 0} G'(k) P_{-k}(\zeta_{t/d} \ge z) \\
	&\le 2E[G(\zeta_{t/d} - z)1(\zeta_{t/d} \ge z)] + E[G'(\zeta_{t/d} - z)1(\zeta_{t/d} \ge z)]. 
\end{aligned}
\end{equation}
Also from Lemma \ref{sum2exp} and \eqref{eq:hatG'<G}, 
\begin{equation}\label{eq:sumhatGbound}
\begin{aligned}
	\sum_{k \ge 0} \widehat G'(k) P_{-k}(\zeta_{t/d} \ge z) &\le E[\widehat G(\zeta_{t/d} - z)1(\zeta_t \ge z)] + E[\widehat G'(\zeta_{t/d} - z)1(\zeta_{t/d} \ge z)] \\
	&\le E[\widehat G(\zeta_{t/d} - z)1(\zeta_t \ge z)] + (d-1)E[G(\zeta_{t/d} - z)1(\zeta_{t/d} \ge z)]. 
\end{aligned}
\end{equation}

Finally, \eqref{eq:Cboundinsums}, \eqref{eq:sumGbound}, \eqref{eq:sumhatGbound}, and $\gamma_{d+1}(t) \le \gamma_d(t)$ yield 
\begin{align*}
	\sC_t(\{g_i\}, z) 
	&\le C \gamma_d(t) ( E[G(\zeta_{t/d} - z)1(\zeta_{t/d} \ge z)] + E[G'(\zeta_{t/d} - z)1(\zeta_{t/d} \ge z)] )^2 \\
	&\quad + C \gamma_{d+1}(t) E_0[G(\zeta_{t/d} - z)1(\zeta_{t/d} \ge z)] E[\widehat G(\zeta_{t/d} - z)1(\zeta_{t/d} \ge z)]. \qedhere
\end{align*}
\end{proof}

After the following two lemmas, we provide the proof of Corollary \ref{Ccor}. Recall the assumption that $\{g_i\}$ satisfy all Conditions \ref{nondecrcond}--\ref{smallderiv}. Recall also that this implies that $G$ satisfies these conditions as well (Remark \ref{G-rmk}). Moreover, we assume that we have a scaling sequence $z$ such that  $\sup_t \E_{\{g_i\}}[N_t] < \infty$, which implies $t^{-1/2}z \to \infty$ (Remark \ref{z/sqrtt}). 

\begin{lemma}\label{Gproperty} Suppose $g_2, \ldots, g_d$ satisfy Conditions \ref{nondecrcond}--\ref{smallderiv}. Then there is $C \in (0,\infty)$ so that $G(j+1) \le C G(j)$ for all nonnegative integers $j$. 
\end{lemma}

\begin{proof} We need only consider the case where $G$ is not a constant function. 
From Condition \ref{covcond}, if $G'(j) = 0$ for an integer $j$, then $G(j-1) = G(j)$. Then from Condition \ref{nondecrcond}, $G$ is constant on $[j-1,j]$, which means $G'(j-1) = 0$. Iterating this argument, we see that $G$ is constant on $[0,j]$. 
As $G$ is nonconstant, there must be $J$ large enough so that $G'(j) > 0$ for all integers $j \ge J$. 

Then by Condition \ref{covcond}, there is $C>0$ so that $G(j+1) - G(j) \le C G'(j+1)$ for all $j \ge J$. Hence,
$G'(j+1)/G(j+1) \ge (1/C)(1 - G(j)/G(j+1))$
for $j$ sufficiently large. It follows from this inequality and Condition \ref{smallderiv} that 
$\limsup_{j \to \infty} G(j+1)/G(j) \le \limsup_{j \to \infty} ( 1 - C G'(j+1)/G(j+1))^{-1} = 1$.
\end{proof}

\begin{lemma}\label{GhatG0} 
Suppose $g_2, \ldots, g_d$ satisfy Conditions \ref{nondecrcond}--\ref{smallderiv} and that  $\sup_{t \ge 0} \E_{\{g_i\}}[N_t] < \infty$. Then, 
\begin{equation}\label{eq:intGbound}
	 \sup_{t\ge0} E_0\Big[ \int_0^{(\zeta_{t/d} - z)_+} G(u)\,du \Big] < \infty. 
\end{equation}
\end{lemma}

\begin{proof} 
Note that
\begin{equation}\label{eq:Gfloorbound}
	G(j) = \prod_{i=2}^d (2g_i(j) + 1) \le \prod_{i=2}^d (2\lfloor g_i(j) \rfloor + 3) \le 3^{d-1} \prod_{i=2}^d (2\lfloor g_i(j) \rfloor + 1). 
\end{equation}
From Lemma \ref{Gproperty} and since $G$ is nondecreasing, $\sup_{u \in [j, j+1]} G(u) = G(j+1) \le CG(j)$ for all $j \in \{0, 1, 2, \ldots\}$.
Using this, followed by \eqref{eq:Gfloorbound} and \eqref{eq:generalNexpression}, we have 
\begin{align*}
	&E_0\Big[ \int_0^{(\zeta_{t/d} - z)_+} G(u)\,du \Big] = \int_0^\infty G(u)P_0(\zeta_{t/d} > z + u)\,du 
	= \sum_{j \ge 0} \int_j^{j+1} G(u) P_0(\zeta_{t/d} > z + u)\,du \\
	&\le  C \sum_{j \ge 0} G(j) P_0(\zeta_{t/d} > z + j) 
	\le C'  \sum_{j \ge 0} \Big(\prod_{i=2}^d (2\lfloor g_i(j) \rfloor + 1)\Big) P_0(\zeta_{t/d} > z + j) = C'\E_{\{g_i\}}[N_t]. 
\end{align*} 
As $\sup_t \mathbb{E}_{\{g_i\}}[N_t]<\infty$, this completes the proof. 
\end{proof}

Lastly, we prove Corollary \ref{Ccor}. 

\begin{proof}[Proof of Corollary \ref{Ccor}] Condition \ref{smallderiv} implies that $G'(u) \le CG(u)$ for some $C$ and all $u$. Thus, 
	$E_0[G'(\zeta_{t/d} - z)1(\zeta_{t/d} \ge z)] \le C E_0[G(\zeta_{t/d} - z)1(\zeta_{t/d} \ge z)]$.
Moreover, \eqref{eq:hatG'<G} and \eqref{eq:intGbound} imply 
	$\sup_{t \ge 0} E_0[\widehat G(\zeta_{t/d} - z)1(\zeta_{t/d} \ge z)] < \infty$. 
Inserting these last two bounds into the estimate of Proposition \ref{mainbound} gives 
\[
	\sC_t(\{g_i\},z) \le C \big( \gamma_d(t) ( E_0[ G(\zeta_{t/d} - z) 1(\zeta_{t/d} \ge z)] )^2 + \gamma_{d+1}(t)  E_0[ G(\zeta_{t/d} - z) 1(\zeta_{t/d} \ge z)] \big). 
\]

Now, if $H = \int_0^u G(v)\,dv$ and $\tilde H(u) = G(u)$, then from Condition \ref{smallderiv}, $\tilde H'(u)/H'(u) \to 0$ as in \eqref{eq:G'/G}. Since $\sup_t\E_{\{g_i\}}[N_t] < \infty$ implies $t^{-1/2}z \to \infty$, from \eqref{eq:intGbound} and Lemma \ref{smallerfunction} we conclude 
$\lim_{t\to\infty} E_0[G(\zeta_{t/d} - z)1(\zeta_{t/d} \ge z)] = \lim_{t\to\infty} E_0[\tilde H(\zeta_{t/d} - z)1(\zeta_{t/d} \ge z)] = 0$,
completing the proof. 
\end{proof}

\section{Appendix}\label{rwlemmas}

Here we give several results regarding a continuous time simple random walk $\{\zeta_t\}$ in one dimension, followed by an asymptotic lemma for the standard Gaussian distribution.  For the reader's convenience, these are collected here, although forms of the results might be found in the literature.

\begin{lemma}\label{sum2exp} For continuous $h: \R_+ \to \R_+$, let $H(u) = \int_0^u h(x)\,dx$. 
If  $h$ is nondecreasing, then for any $z \in \R$, 
\[
	\Big|E_0[H(\zeta_t - z)1(\zeta_t > z)] - \sum_{j \ge 0} h(j) P_{-j}(\zeta_t  > z)\Big| \le E_0[h(\zeta_t - z)1(\zeta_t > z)]. 
\]
\end{lemma}

\begin{proof} 
Since $h$ is continuous, $H$ is differentiable and 
$E_0[H(\zeta_t -z)1(\zeta_t > z)] =  \int_0^\infty h(u) P_0(\zeta_t > z + u)\,du$.

Note also the following. Let $j \in \Z$ and $u \in (j, j+1)$. If $z \not\in \Z$, then 
\begin{align*}
	P_0(\zeta_{t} > z + j) 
	& = P_0(\zeta_{t} = \lceil z \rceil + j) + P_0(\zeta_{t} \ge \lceil z \rceil + j + 1) \\
	&\le P_0(\zeta_{t} = \lceil z \rceil + j) + P_0(\zeta_{t} > z + u). 
\end{align*}
Then, the nondecreasing property of $h$ implies 
\begin{align*}
	\sum_{j \ge 0} h(j) P_{-j}(\zeta_t  > z) &= \sum_{j \ge 0} \int_j^{j+1} h(j) P_0(\zeta_{t} > z + j)\,du \\
	&\le \sum_{j \ge 0} \int_j^{j+1} h(u) P_0(\zeta_{t} > z + u)\,du + \sum_{j \ge 0} h(j) P(\zeta_{t} = \lceil z \rceil + j) \\
	&= \int_0^\infty h(u) P_0(\zeta_{t} > z + u)\,du + E_0[h(\zeta_{t} - \lceil z \rceil) 1(\zeta_{t} \ge \lceil z \rceil)] \\
	&\le E_0[H(\zeta_{t} - z)1(\zeta_t > z)] + E_0[h(\zeta_t - z)1(\zeta_t > z)]. 
\end{align*}
Otherwise if $z \in \Z$, 
$P_0(\zeta_{t/d} > z + j) = P_0(\zeta_{t/d} \ge z + j + 1) \le P_0(\zeta_{t/d} > z + u)$,
and we obtain 
$\sum_{j \ge 0} h(j) P_{-j}(\zeta_t  > z)  \le E_0[H(\zeta_{t} - z)1(\zeta_t > z)]$.

For a lower bound, for any $z$ we have 
\begin{equation}
\label{eq:lbpart1}
\begin{aligned}
	\sum_{j \ge 0} h(j) P_{-j}(\zeta_t  > z) &\geq \sum_{j \ge 1} \int_{j}^{j+1} h(j) P_0(\zeta_t > z + j)\,du 
	\ge\sum_{j \ge 1} \int_j^{j+1} h(u-1)P_0(\zeta_t > z + u)\,du  \\
	&= \int_1^\infty h(u-1)P_0(\zeta_t > z + u)\,du 
	= \int_0^\infty h(u)P_0(\zeta_t > z + u + 1)\,du. 
\end{aligned}
\end{equation}
Moreover, 
\begin{align*}
	&E_0[H(\zeta_t - z)1(\zeta_t > z)] - \int_0^\infty h(u)P_0(\zeta_t > z + u + 1)\,du \\
	&= \int_0^\infty h(u) (P_0(\zeta_t > z + u) - P_0(\zeta_t > z + u + 1))\,du \\
	&= \int_0^\infty h(u) P_0(u < \zeta_t - z \le u + 1)\,du
	= E_0\Big[ \int_0^\infty h(u) 1(\zeta_t - z - 1 \le u < \zeta_t - z)\,du \Big] \\
	&= E_0\Big[ \int_{(\zeta_t - z - 1)_+}^{\zeta_t - z} h(u)\,du\, 1(\zeta_t > z) \Big] \\
	&\le E_0[ h(\zeta_t - z)(\zeta_t - z - (\zeta_t - z - 1)_+)1(\zeta_t > z)] \le E_0[h(\zeta_t - z)1(\zeta_t > z)]. 
\end{align*}
The previous display, with \eqref{eq:lbpart1}, gives 
$\sum_{j \ge 0} h(j) P_{-j}(\zeta_t  > z) \ge E_0[H(\zeta_t - z)1(\zeta_t > z)] - E_0[h(\zeta_t - z)1(\zeta_t > z)]$,
completing the proof. 
\end{proof}

\begin{lemma}\label{smallerfunction} Let $H, \tilde H : \R_+ \to \R_+$ be nondecreasing and continuously differentiable with $h = H'$ and $\tilde h = \tilde H'$ satisfying $\lim_{u\to\infty} \tilde h(u)/h(u) = 0$. 
If $t^{-1/2}z \to \infty$ as $t \to \infty$ and 
$\sup_{t \ge 0} E_0[H(\zeta_t - z)1(\zeta_t > z)] < \infty$,
then $\lim_{t\to\infty} E_0[\tilde H(\zeta_t - z)1(\zeta_t > z)] = 0$. 
\end{lemma}

\begin{proof} 
For any $\eps > 0$, there is $u_\eps \ge 0$ so $\tilde h(u) \le \eps h(u)$ for $u \ge u_\eps$. Then, 
\begin{align*}
	E_0[\tilde H(\zeta_t - z)1(\zeta_t > z)] &= \tilde H(0) P_0(\zeta_t > z) + \int_0^\infty \tilde h(u)P_0(\zeta_t > z + u)\,du \\
	&\le \Big( \tilde H(0) + u_\eps\sup_{u \le u_\eps}\tilde h(u) \Big) P_0(\zeta_t > z) + \eps \int_{u_\eps}^\infty h(u)P_0(\zeta_t > z + u)\,du \\
	&\le \Big( \tilde H(0) + u_\eps\sup_{u \le u_\eps}\tilde h(u)  
	\Big) P_0(\zeta_t > z) 
	+ \eps E_0[H(\zeta_t - z)1(\zeta_t > z)]. 
\end{align*}
Moreover, $t^{-1/2}z \to \infty$ implies $P_0(\zeta_t > z) \to 0$, by the central limit theorem. 

Thus, 
$\limsup_{t\to\infty} E_0[\tilde H(\zeta_t - z)1(\zeta_t > z)] \le \eps \sup_{t \ge 0} E_0\left[H(\zeta_t - z)1(\zeta_t > z)\right]$,
from which the result follows by letting $\eps \to 0$.  
\end{proof}

\begin{lemma}\label{jsum} Let $h : \R_+ \to \R_+$ be a function satisfying $E_j[h(-\zeta_t)1(\zeta_t \le 0)] < \infty$ for all $j \in \Z$ and $t \ge 0$. Then for any $0 < s < t$ and $z \in \R$, 
\[
	\sum_{j \in \Z} E_j[h(-\zeta_s)1(\zeta_s \le 0)] P_j(\zeta_{t-s} \ge z) = \sum_{k \le 0} h(-k) P_k(\zeta_t \ge z). 
\]
If in addition $h$ is nondecreasing, then 
\[
	 \sup_{0 < s < t} \sup_{j \in \Z} E_j[h(-\zeta_s)1(\zeta_s \le 0)] P_j(\zeta_{t-s} \ge z) \le E_0[h(\zeta_t - z)1(\zeta_t \ge z)]. 
\]
\end{lemma}

\begin{proof} These are computations using the properties $P_0(- \zeta_t \in \cdot) = P_0(\zeta_t \in \cdot)$ and $P_j(\zeta_t \in \cdot) = P_0(\zeta_t + j \in \cdot)$. 
First, 
\begin{align*}
	&\sum_{j \in \Z} E_j[h(-\zeta_s)1(\zeta_s \le 0)] P_j(\zeta_{t-s} \ge z) = \sum_{j \in \Z} E_0[h(-\zeta_s - j)1(\zeta_s + j \le 0)] P_j(\zeta_{t-s} \ge z) \\
	&= \sum_{j\in\Z} \sum_{k \le 0} h(-k) P_0(\zeta_s + j = k) P_j(\zeta_{t-s} \ge z) 
	= \sum_{k \le 0} h(-k) \sum_{j \in \Z} P_0(\zeta_s + k = j) P_j(\zeta_{t-s} \ge z) \\
	&= \sum_{k \le 0} h(-k) \sum_{j \in \Z} P_k(\zeta_s = j) P_j(\zeta_{t-s} \ge z) 
	= \sum_{k \le 0} h(-k) P_k(\zeta_t \ge z). 
\end{align*}
The last line above follows from the Chapman-Kolmogorov equation for the process $\zeta_t$. 

Next, suppose that $h$ is nondecreasing. We have 
\begin{align*}
&E_j[h(-\zeta_{s})1(\zeta_{s} \le 0)] P_j(\zeta_{t-s} \ge z) 
	= E_j[h(-\zeta_{s})1(\zeta_{s} \le 0)] P_j(\zeta_{t-s} \ge \lceil z \rceil) \\
	&= E_0[ h(- \zeta_{s} - j)1(\zeta_{s} \le -j)] P_0(\zeta_{t-s} \ge \lceil z \rceil - j) 
	= E_0[ h(\zeta_{s} - j)1(\zeta_{s} \ge j)] P_0(\zeta_{t-s} \ge \lceil z \rceil - j) \\
	&= E_{\lceil z \rceil - j}[h(\zeta_{s} - \lceil z \rceil)1(\zeta_{s} \ge \lceil z \rceil)]P_0(\zeta_{t-s} \ge \lceil z \rceil - j) . 
\end{align*}
Hence, 
$$\sup_{j \in \Z} E_j[h(-\zeta_{s})1(\zeta_{s} \le 0)]P_j(\zeta_{t-s} \ge z) 
	= \sup_{k \in \Z} E_{k}[h( \zeta_s - \lceil z \rceil )1(\zeta_{s} \ge \lceil z \rceil)]P_0(\zeta_{t-s} \ge k).$$	
Now, $\zeta_{t-s} \overset{d}{=} \zeta_t - \zeta_s$ and $\zeta_t - \zeta_s$ is independent of $\zeta_s$. This implies 
\begin{align*}
	&E_{k}[h( \zeta_s -  \lceil z \rceil )1(\zeta_{s} \ge \lceil z \rceil)]P_0(\zeta_{t-s} \ge k) = E_{0}[h( \zeta_s + k -  \lceil z \rceil )1(\zeta_{s} + k\ge \lceil z \rceil )] P_0(\zeta_{t-s} \ge k) \\
	&= E_{0}[h( \zeta_s + k -  \lceil z \rceil  )1( \zeta_t - \zeta_s \ge k, \zeta_{s} + k \ge \lceil z \rceil )] 
	= E_0[h(\zeta_s + k - \lceil z \rceil) 1( \lceil z \rceil \le \zeta_s + k \le \zeta_t)] \\
	&\le E_0[h(\zeta_t - \lceil z \rceil)1(\zeta_t \ge \lceil z \rceil)] 
	\le E_0[h(\zeta_t - z)1(\zeta_t \ge z)], 
\end{align*}
where monotonicity of $h$ was used in the last two lines. 
The result follows. 
\end{proof}

For the next lemma, recall that $x(t) \sim y(t)$ denotes $\lim_{t\to\infty} x(t)/y(t) = 1$, and that $X$ is standard Gaussian with density function $\varphi$. 
\begin{lemma}\label{rw2normal}  
Suppose $H : \R_+ \to \R_+$ is continuously differentiable with $h = H'$ satisfying 
\begin{equation}\label{eq:polygrowthass}
	h(u) \le C(1 + u^{\beta-1}), \qquad   u \ge 0, 
\end{equation}
for some $C > 0$ and $\beta \ge 1$. If 
$z = o(t^{2/3})$ as $t \to \infty$, 
then 
\[
	E_0\left[H(\zeta_t - z)1(\zeta_t > z)\right] \sim E[H(\sqrt{t}X - z)1(\sqrt{t}X > z)], \qquad t \to \infty. 
\]
\end{lemma}

\begin{proof} We make use of the following large deviation result, which can be found on page 552 of \cite{FellBook71}: There is $C > 0$ so that, if $u = o(t^{1/6})$ and $t$ is sufficiently large, 
\begin{equation}\label{eq:largedev}
	\Big| \frac{P_0(\zeta_t > u\sqrt{t})}{P(X > u)} - 1 \Big| \le \frac{Cu^3}{\sqrt{t}}. 
\end{equation}
In particular, 
\begin{equation}\label{eq:tailequiv}
	P_0(\zeta_t > u\sqrt{t}) \sim P(X > u), \qquad t\to\infty, \quad \text{when} \quad u = o(t^{1/6}). 
\end{equation}

Then since $P_0(\zeta_t > z) - P(\sqrt{t}X > z) \to 0$,  
\begin{align*}
	&E_0[H(\zeta_t - z)1(\zeta_t > z)] = H(0) P_0(\zeta_t > z) + \int_0^\infty h(u)P_0(\zeta_t > z + u)\,du \\
	&= E[H(\sqrt{t}X - z)1(\sqrt{t}X > z)] + H(0) (P_0(\zeta_t > z) - P(\sqrt{t}X > z) ) \\
	&\quad +  \int_0^\infty h(u)P_0(\zeta_t > z + u)\,du - \int_0^\infty h(u)P_0(\sqrt{t}X > z + u)\,du \\
	&\sim E[H(\sqrt{t}X - z)1(\sqrt{t}X > z)]  \\
	&\quad +  \int_0^\infty h(u)P_0(\zeta_t > z + u)\,du - \int_0^\infty h(u)P_0(\sqrt{t}X > z + u)\,du. 
\end{align*}
Hence we show that 
\begin{equation}\label{eq:reducedasym}
	\int_0^\infty h(u)P_0(\zeta_t > z + u)\,du \sim \int_0^\infty h(u)P_0(\sqrt{t}X > z + u)\,du, \qquad t \to \infty, 
\end{equation}
which will imply the result. 

For notational convenience, let $w = t^{-1/2} z$. Let $r \to \infty$ denote a sequence such that $r = o(t^{1/6})$ and $r/w \to \infty$ as $t \to \infty$ (for example, take $r = w|\log(t^{-1/6}w)|$ and note that by the assumption $z = o(t^{2/3})$, we have $w = o(t^{1/6})$). 
Then we may write 
\begin{equation}\label{eq:rlogt}
\begin{aligned}
	\int_0^\infty h(u)P_0(\zeta_t > z + u)\,du &= \sqrt{t} \int_w^\infty h(\sqrt{t} u - z)P_0(t^{-1/2} \zeta_t > u)\,du \\
	&= \sqrt{t} \int_w^{r \vee \log t}  h(\sqrt{t} u - z)P_0(t^{-1/2} \zeta_t > u)\,du + o(1), 
\end{aligned}
\end{equation}
which is justified as follows. 

Using \eqref{eq:polygrowthass} in the third line, the Cauchy-Schwarz inequality in the sixth line, and \eqref{eq:tailequiv} along with $E_0[\zeta_t^{2\beta}] \leq C t^{\beta}$ (e.g., Burkholder-Davis-Gundy inequality) in the last line, 
\begin{align}\nonumber
	&\sqrt{t} \int_{r \vee \log t}^\infty h(\sqrt{t}u - z)P_0(t^{-1/2} \zeta_t > u)\,du \\ \nonumber
	&\le \int_0^\infty h(u + \sqrt{t} \log t - z)P_0(\zeta_t > u + \sqrt{t}\log t)\,du \\ \nonumber
	&\le C \int_0^\infty \left( 1 + (u + \sqrt{t}\log t)^{\beta-1} \right) P_0(\zeta_t - \sqrt{t}\log t > u)\,du \\ \nonumber
	&\le C' \Big(\int_0^\infty u^{\beta-1}P_0(\zeta_t - \sqrt{t}\log t > u)\,du + t^{(\beta-1)/2}(\log t)^{\beta-1} \int_0^\infty P_0(\zeta_t - \sqrt{t}\log t > u)\,du \Big) \\ \nonumber
	&\le C' \left( E_0[\zeta_t^\beta 1(\zeta_t > \sqrt{t}\log t)] + t^{(\beta-1)/2}(\log t)^{\beta-1} E_0[\zeta_t 1(\zeta_t > \sqrt{t}\log t)] \right) \\ \nonumber
	&\le C' \left( (E_0[\zeta_t^{2\beta}])^{1/2} + t^{(\beta-1)/2}(\log t)^{\beta-1} (E_0[\zeta_t^2])^{1/2}  \right) P_0(\zeta_t > \sqrt{t}\log t)^{1/2} \\ \label{eq:afterlog}
	&\leq C'' t^{\beta/2} (\log t)^{\beta-1} P(X > \log t)^{1/2} \le C''' t^{\beta/2}(\log t)^{\beta-1} e^{-(1/4)(\log t)^2} \to 0, 
\end{align}
as $t \to \infty$. 

Continuing from \eqref{eq:rlogt}, we write 
\begin{equation}\label{eq:gaussiandecomp}
\begin{aligned} 
	&\int_w^{r \vee \log t}  h(\sqrt{t} u - z)P_0(t^{-1/2} \zeta_t > u)\,du \\
	 &= \int_w^{r \vee \log t}  h(\sqrt{t} u - z) P(X > u)\,du +  \int_w^{r \vee \log t}  h(\sqrt{t} u - z)\Big( \frac{P_0( \zeta_t > u\sqrt{t})}{P(X > u)} - 1 \Big) P(X > u)\,du. 
\end{aligned}
\end{equation}
Since $r \vee \log t = o(t^{1/6})$, \eqref{eq:largedev} holds uniformly in $u \in [w, r\vee\log t]$ and in $t$ sufficiently large. 
Applying this to the last term in \eqref{eq:gaussiandecomp}, we have 
\begin{align} \nonumber
	& \Big|\sqrt{t}\int_w^{r \vee \log t}  h(\sqrt{t} u - z)\Big( \frac{P_0( \zeta_t > u\sqrt{t})}{P(X > u)} - 1 \Big) P(X > u)\,du\Big| \\  \nonumber
	&\le C\int_w^{r \vee \log t}  u^3 h(\sqrt{t} u - z) P(X > u)\,du \\ \label{eq:w<logt}
	&\le C1(w \le \log t) \int_w^{\log t} u^3 h(\sqrt{t} u - z) P(X > u)\,du \\ \label{eq:logt<r}
	&\quad + C1(\log t \le r) \int_{\log t}^r u^3 h(\sqrt{t} u - z) P(X > u)\,du. 
\end{align}
 For \eqref{eq:w<logt}, we have 
\[
	1(w \le \log t) \int_w^{\log t} u^3 h(\sqrt{t} u - z) P(X > u)\,du \le (\log t)^3 \int_w^{r \vee \log t} h(\sqrt{t}u - z)P(X > u)\,du, 
\]
and, using \eqref{eq:polygrowthass} and Lemma \ref{lem:normal}, for \eqref{eq:logt<r} we have 
\begin{align*}
	&1(\log t \le r) \int_{\log t}^r u^3 h(\sqrt{t} u - z) P(X > u)\,du 
	\le 2C t^{(\beta-1)/2} \int_{\log t}^\infty u^{\beta+2} P(X > u)\,du \\
	& \le  C' t^{(\beta-1)/2} E[(X - \log t)_+^{\beta+3}] \sim  C'' t^{(\beta-1)/2}\frac{\varphi(\log t)}{(\log t)^{\beta+4}} = O\Big( \frac{t^{(\beta-1)/2}}{(\log t)^{\beta+4}} e^{-(\log t)^2/2} \Big). 
\end{align*}

Thus from \eqref{eq:gaussiandecomp}, 
\begin{align} \nonumber
	 &\sqrt{t} \int_w^{r \vee \log t}  h(\sqrt{t} u - z)P_0(t^{-1/2} \zeta_t > u)\,du \\ \nonumber
	 &= \sqrt{t} \Big( 1 + O\Big(\frac{(\log t)^3}{\sqrt{t}}\Big) \Big) \int_w^{r \vee \log t}  h(\sqrt{t} u - z) P(X > u)\,du + O\Big( \frac{t^{\beta/2}}{(\log t)^{\beta+4}} e^{-(\log t)^2/2} \Big)\\ \label{eq:togaussian}
	 &\sim \sqrt{t} \int_w^{r \vee \log t}  h(\sqrt{t} u - z) P(X > u)\,du, \qquad t\to\infty. 
\end{align}
Combining \eqref{eq:rlogt} and \eqref{eq:togaussian}, we have shown 
\[
	\int_0^\infty h(u)P_0(\zeta_t > z + u)\,du \sim \sqrt{t} \int_w^{r \vee \log t}  h(\sqrt{t} u - z) P(X > u)\,du, \qquad t\to\infty.
\]

Lastly, \eqref{eq:reducedasym} is proved upon noting that
\begin{align*}
	\sqrt{t} \int_w^{r \vee \log t}  h(\sqrt{t} u - z) P(X > u)\,du &\sim \sqrt{t} \int_w^{\infty}  h(\sqrt{t} u - z) P(X > u)\,du \\
	&= \int_0^\infty h(u) P(\sqrt{t}X > z + u)\,du, 
\end{align*}
which follows from repeating the arguments culminating in \eqref{eq:afterlog} with $t^{-1/2}\zeta_t$ replaced by $X$:
\begin{align*}
	&\sqrt{t} \int_{r\vee\log t}^\infty h(\sqrt{t} u - z) P(X > u)\,du \\
	&\le C (t^{\beta/2} E[X^\beta 1(X > \log t)] + t^{\beta/2}(\log t)^{\beta-1} E[X 1(X > \log t)])  \\
	&\le  C' t^{\beta/2}(\log t)^{\beta-1}P(X > \log t)^{1/2}\to 0, \qquad t\to\infty. \qedhere
\end{align*}
\end{proof}

\begin{lemma}\label{lem:normal} Let $X$ be a standard Gaussian random variable and let $\varphi(u) = (2\pi)^{-1/2}e^{-u^2/2}$. For any $\beta \ge 0$, 
\[
	E[ (X - u)_+^\beta] = \frac{\Gamma(\beta+1) \varphi(u)}{u^{\beta+1}} + O\Big( \frac{\varphi(u)}{u^{\beta+3}} \Big), \qquad u \to \infty. 
\]
\end{lemma}

\begin{proof} We have 
\begin{align*}
	\frac{u^{\beta+1}}{\varphi(u)} E[ (X - u)_+^\beta] &= \frac{u^{\beta+1}}{\varphi(u)} \int_u^\infty (x-u)^\beta \varphi(x)\,dx = u^{\beta+1} \int_0^\infty x^\beta \frac{\varphi(x+u)}{\varphi(u)}\,dx \\
	&= \int_0^\infty (ux)^\beta e^{-x^2/2 - ux}\,d(ux) = \int_0^\infty x^\beta e^{-(x/u)^2/2 - x}\,dx \\
	&= \Gamma(\beta + 1) + \int_0^\infty x^{\beta} e^{-x} ( e^{-(x/u)^2/2} - 1)\,dx. 
\end{align*}
Using $0 \le 1 - e^{-y} \le y$ for $y \ge 0$, 
\begin{align*}
	0 \le \Gamma(\beta + 1) - \frac{u^{\beta+1}}{\varphi(u)} E[ (X - u)_+^\beta] &= \int_0^\infty x^\beta e^{-x} (1 -  e^{-(x/u)^2/2} ) \,dx 
	\le \frac{\Gamma(\beta + 3)}{2u^2}. \qedhere
\end{align*}
\end{proof}

\end{document}